\documentclass[11pt,leqno]{article}
\usepackage{amsfonts}
\usepackage{amsmath}
\usepackage{amssymb}
\usepackage{amsthm}
\usepackage{latexsym}
\usepackage{mathrsfs}
\usepackage{amsthm}
\usepackage{mathtools}
\usepackage{bm}
\usepackage{verbatim}

\newtheorem{theorem}{\bf Theorem}[section]
\newtheorem{lemma}[theorem]{\bf Lemma}

\newtheorem{proposition}[theorem]{\bf Proposition}
\newtheorem{assumption}[theorem]{\bf Assumption}
\newtheorem{remark}[theorem]{\bf Remark}
\newtheorem{definition}[theorem]{\bf Definition}

\renewcommand{\thefootnote}{\fnsymbol{footnote}}

\makeatletter
\renewenvironment{proof}[1][\proofname]{\par
  \pushQED{\qed}%
  \normalfont \topsep6\p@\@plus6\p@\relax
  \trivlist
  \item\relax
  {
  #1\@addpunct{.}}\hspace\labelsep\ignorespaces
}{%
  \popQED\endtrivlist\@endpefalse
}
\makeatother

\numberwithin{equation}{section}

\newcommand{\dt}{\partial_t}
\newcommand{\ds}{\partial_s}

\makeatletter
\newcommand{\opnorm}{\@ifstar\@opnorms\@opnorm}
\newcommand{\@opnorms}[1]{%
  \left|\mkern-1.5mu\left|\mkern-1.5mu\left|
   #1
  \right|\mkern-1.5mu\right|\mkern-1.5mu\right|
}
\newcommand{\@opnorm}[2][]{%
  \mathopen{#1|\mkern-1.5mu#1|\mkern-1.5mu#1|}
  #2
  \mathclose{#1|\mkern-1.5mu#1|\mkern-1.5mu#1|}
}

\begin{document}
\vspace*{0ex}
\begin{center}
{\Large\bf
Well-posedness of the initial boundary value problem \\
for degenerate hyperbolic systems with a localized term \\
and its application to the linearized system \\[0.5ex]
for the motion of an inextensible hanging string
}
\end{center}

\begin{center}
Tatsuo Iguchi and Masahiro Takayama\footnote{Corresponding author}
\end{center}

\renewcommand{\thefootnote}{\fnsymbol{footnote}}
\renewcommand{\thefootnote}{\arabic{footnote}}

\begin{abstract}
Motivated by an analysis on the well-posedness of the initial boundary value problem for the motion of an inextensible hanging string, 
we first consider an initial boundary value problem for one-dimensional degenerate hyperbolic systems with a localized term 
and show its well-posedness in weighted Sobolev spaces. 
We then consider the linearized system for the motion of an inextensible hanging string. 
Well-posedness of its initial boundary value problem is demonstrated as an application of the result obtained in the first part. 
\end{abstract}

\section{Introduction}\label{sect:intro}
The present paper consists of two parts. 
In the first part, motivated by an analysis on the well-posedness of the initial boundary value problem for the motion of an inextensible hanging string, 
we consider the initial boundary value problem 
\begin{equation}\label{LS}
\begin{cases}
 \ddot{\bm{u}}=(A(s,t)\bm{u}')'+Q(s,t)\bm{u}'(1,t)+\bm{f}(s,t) &\mbox{in}\quad (0,1)\times(0,T), \\
 \bm{u}=\bm{0} &\mbox{on}\quad \{s=1\}\times(0,T), \\
 (\bm{u},\dot{\bm{u}})|_{t=0}=(\bm{u}_0^\mathrm{in},\bm{u}_1^\mathrm{in}) &\mbox{in}\quad (0,1),
\end{cases}
\end{equation}
where $\bm{u}$ is a $\mathbb{R}^N$-valued unknown function of $(s,t)\in[0,1]\times[0,T]$, while $\bm{f}$, $\bm{u}_0^\mathrm{in}$, 
and $\bm{u}_1^\mathrm{in}$ are $\mathbb{R}^N$-valued given functions, and $A$ and $Q$ are $N\times N$ matrix valued given functions. 
Here, $\dot{\bm{u}}$ and $\bm{u}'$ denote derivatives of $\bm{u}$ with respect to $t$ and $s$, respectively. 
Moreover, we assume that $A(s,t)$ is symmetric and satisfies $A(s,t)\simeq s\mathrm{Id}$, where $\mathrm{Id}$ is an identity matrix. 
Therefore, the coefficient matrix $A(s,t)$ degenerates at one end $s=0$ of the interval, 
so that the first equation in \eqref{LS} is a linear degenerate hyperbolic system with a localized term. 
Due to this degeneracy, we do not need to impose any boundary conditions on this end $s=0$. 
The first objective in this paper is to establish the well-posedness of this initial boundary value problem.

\medskip
One of difficulties of this problem comes from the degeneracy of the coefficient matrix $A(s,t)$. 
This type of degenerate hyperbolic systems in the analysis of the motion of strings has already been analyzed by several authors, for example, 
Koshlyakov, Gliner, and Smilnov \cite{KoshlyakovGlinerSmilnov1962}, Reeken \cite{Reeken1979-1, Reeken1979-2}, Yamaguchi \cite{Yamaguchi2008}, 
Preston \cite{Preston2011}, and Takayama \cite{Takayama2018}, 
and the difficulty has been overcame by using appropriate weighted Sobolev spaces. 
In the present paper, we adopt the weights used by Reeken \cite{Reeken1979-1, Reeken1979-2} and Takayama \cite{Takayama2018}. 
We note that weighted Sobolev spaces have been efficiently used also in the analysis of degenerate hyperbolic systems appearing in the fluid mechanics, for example, 
in the analysis of the nonlinear shallow water and Green--Naghdi equations by Lannes and M\'{e}tivier \cite{LannesMetivier2018} for the motion of water surface 
near the shoreline where the depth of the water vanishes, and in the analysis of the Euler--Poisson equations by Makino \cite{Makino2015} for the motion of 
a gaseous star surrounded by a free surface where the density and the pressure of the gas vanish. 
Another difficulty of the problem comes from the localized term $Q(s,t)\bm{u}'(1,t)$, which cannot be regarded as a lower order term. 
However, as we will see in this paper, by introducing an appropriate energy functional we obtain an a priori energy estimate for the solution $\bm{u}$. 
Although such an energy estimate is crucial to show the well-posedness of the problem, it does not imply directly the existence of the solution. 
Our idea showing the existence of the solution is to regularize the hyperbolic system as 
\begin{equation}\label{rLS}
\begin{cases}
 \ddot{\bm{u}}=(A(s,t)\bm{u}')'+Q(s,t)\bm{u}'(1,t)+\varepsilon s\dot{\bm{u}}'+\bm{f}(s,t) &\mbox{in}\quad (0,1)\times(0,T), \\
 \bm{u}=\bm{0} &\mbox{on}\quad \{s=1\}\times(0,T), \\
 (\bm{u},\dot{\bm{u}})|_{t=0}=(\bm{u}_0^{\mathrm{in},\varepsilon},\bm{u}_1^{\mathrm{in},\varepsilon}) &\mbox{in}\quad (0,1),
\end{cases}
\end{equation}
where $\varepsilon>0$ is a regularizing parameter and the initial data $(\bm{u}_0^{\mathrm{in},\varepsilon},\bm{u}_1^{\mathrm{in},\varepsilon})$ should be 
modified from the original initial data $(\bm{u}_0^{\mathrm{in}},\bm{u}_1^{\mathrm{in}})$ so that the corresponding compatibility conditions are satisfied. 
Thanks to the regularized term $\varepsilon s\dot{\bm{u}}'$, the solution to this regularized problem has an additional boundary regularity 
so that the localized term $Q(s,t)\bm{u}'(1,t)$ can be regarded as a lower order term. 
As a result, to show the existence of the solution $\bm{u}^\varepsilon$ for $\varepsilon>0$ it is sufficient to consider the case $Q(s,t)=O$. 
In such a case, we can follow the idea used by Takayama \cite{Takayama2018}, that is, 
we transform the problem on the one-dimensional interval $(0,1)$ into a problem on the two-dimensional unit disc $D$ 
by the transformation $\bm{u}^\sharp(x_1,x_2,t)=\bm{u}(x_1^2+x_2^2,t)$. 
Then, the transformed two-dimensional problem forms a non-degenerate hyperbolic system so that the standard theory of hyperbolic systems can be applicable 
to show the existence of the solution $\bm{u}^\varepsilon$ to the regularized problem for $\varepsilon>0$. 
Then, passing to the limit $\varepsilon\to+0$ we obtain the solution $\bm{u}$ to the problem \eqref{LS}. 
To the best of our knowledge, there is no existing result on initial boundary value problems for hyperbolic systems with this type of a localized term. 
Instead, we mention Fukuda and Suzuki \cite{FukudaSuzuki2005} and Okada and Fukuda \cite{OkadaFukuda2003}, 
where an initial boundary value problem for a semilinear parabolic equation with a localized term has been studied.

\medskip
The problem \eqref{LS} arises in the analysis on the well-posedness of the initial boundary value problem 
for the motion of an inextensible hanging string of finite length under the action of the gravity. 
The model of the motion consists of the initial boundary value problem 
\begin{equation}\label{HS}
\begin{cases}
 \ddot{\bm{x}} - (\tau\bm{x}')'=\bm{g} &\mbox{in}\quad (0,1)\times(0,T), \\
 \bm{x}=\bm{0} &\mbox{on}\quad \{s=1\}\times(0,T), \\
 (\bm{x},\dot{\bm{x}})|_{t=0}=(\bm{x}_0^\mathrm{in},\bm{x}_1^\mathrm{in}) &\mbox{in}\quad (0,1)
\end{cases}
\end{equation}
for the position vector $\bm{x}$ of the string coupled with the two-point boundary value problem 
\begin{equation}\label{BVP}
\begin{cases}
 -\tau''+|\bm{x}''|^2\tau = |\dot{\bm{x}}'|^2 &\mbox{in}\quad (0,1)\times(0,T), \\
 \tau=0 &\mbox{on}\quad \{s=0\}\times(0,T), \\
 \tau'=-\bm{g}\cdot\bm{x}' &\mbox{on}\quad \{s=1\}\times(0,T)
\end{cases}
\end{equation}
for the tension $\tau$ of the string, where $\bm{g}$ is the acceleration of gravity vector assumed to be constant. 
For more details on this model, we refer to Iguchi and Takayama \cite{IguchiTakayama2023}, 
where a priori estimates for the solution $(\bm{x},\tau)$ were obtained in weighted Sobolev spaces. 
In the second part of this paper, we consider a linearized system of this problem. 
Let us linearize the problem around $(\bm{x},\tau)$ and denote the variations by $(\bm{y},\nu)$. 
Then, the linearized system has the form 
\begin{equation}\label{LEq}
\begin{cases}
 \ddot{\bm{y}}=(\tau\bm{y}')'+(\nu\bm{x}')'+\bm{f} &\mbox{in}\quad (0,1)\times(0,T), \\
 \bm{y}=\bm{0} &\mbox{on}\quad \{s=1\}\times(0,T), \\
 (\bm{y},\dot{\bm{y}})|_{t=0}=(\bm{y}_0^\mathrm{in},\bm{y}_1^\mathrm{in}) &\mbox{in}\quad (0,1),
\end{cases}
\end{equation}
and 
\begin{equation}\label{LBVP}
\begin{cases}
 -\nu''+|\bm{x}''|^2\nu = 2\dot{\bm{x}}'\cdot\dot{\bm{y}}' - 2(\bm{x}''\cdot\bm{y}'')\tau + h &\mbox{in}\quad (0,1)\times(0,T), \\
 \nu = 0 &\mbox{on}\quad \{s=0\}\times(0,T), \\
 \nu' = -\bm{g}\cdot\bm{y}' &\mbox{on}\quad \{s=1\}\times(0,T),
\end{cases}
\end{equation}
where $\bm{f}$ and $h$ can be regarded as given functions. 
Here, we note that under appropriate assumptions on $(\bm{x},\tau)$, once $\bm{y}$ is given, 
the above two-point boundary value problem for $\nu$ can be solved uniquely.

\medskip
The second objective in this paper is to establish the well-posedness of the problem \eqref{LEq} and \eqref{LBVP} in weighted Sobolev spaces 
by applying the result in the first part of this paper on the well-posedness of the problem \eqref{LS}. 
To this end, we need to figure out the principal term of $\nu$ in terms of ${\bm{y}}$ explicitly because the term $(\nu\bm{x}')'$ in \eqref{LEq} 
cannot be regarded as a lower order term. 
As we will see later, we decompose $\nu$ as a sum of a principal part $\nu_\mathrm{p}$ and a lower order part $\nu_\mathrm{l}$. 
Moreover, the principal part can be written explicitly as 
\begin{equation}\label{pnu}
\nu_\mathrm{p}(s,t) = -((\bm{g}+2\tau\bm{x}'')(1,t)\cdot\bm{y}'(1,t))\phi(s,t),
\end{equation}
where $\phi$ is a unique solution to the two-point boundary value problem 
\begin{equation}\label{defphi}
\begin{cases}
 -\phi''+|\bm{x}''|^2\phi = 0 &\mbox{in}\quad (0,1)\times(0,T), \\
 \phi=0 &\mbox{on}\quad \{s=0\}\times(0,T), \\
 \phi'=1 &\mbox{on}\quad \{s=1\}\times(0,T).
\end{cases}
\end{equation}
Plugging the decomposition $\nu=\nu_\mathrm{p}+\nu_\mathrm{l}$ into \eqref{LEq}, we obtain 
\begin{equation}\label{rLEq}
\begin{cases}
 \ddot{\bm{y}}=(A\bm{y}')'+Q\bm{y}'(1,t)+(\nu_\mathrm{l}\bm{x}')'+\bm{f} &\mbox{in}\quad (0,1)\times(0,T), \\
 \bm{y}=\bm{0} &\mbox{on}\quad \{s=1\}\times(0,T), \\
 (\bm{y},\dot{\bm{y}})|_{t=0}=(\bm{y}_0^\mathrm{in},\bm{y}_1^\mathrm{in}) &\mbox{in}\quad (0,1),
\end{cases}
\end{equation}
where $A(s,t)=\tau(s,t)\mathrm{Id}$ and $Q(s,t)=-(\phi\bm{x}')'(s,t)\otimes(\bm{g}+2\tau\bm{x}'')(1,t)$. 
This problem has the same form as \eqref{LS} so that we can apply the result of the first part. 
However, in order to guarantee that the term $(\nu_\mathrm{l}\bm{x}')'$ is in fact of lower order, 
we need a detailed analysis on a two-point boundary value problem for $\nu_\mathrm{l}$.

\medskip
The contents of this paper are as follows. 
In Section \ref{sect:results} we begin with introducing weighted Sobolev spaces $X^m$ and $Y^m$ for non-negative integers $m$. 
These spaces play an important role in the problems. 
We then state our main results in this paper: well-posedness of the problem \eqref{HS} in Theorem \ref{Th1} 
and that of the problem \eqref{LEq} and \eqref{LBVP} in Theorem \ref{Th2}. 
In Section \ref{sect:pre} we present basic properties of the weighted Sobolev spaces and related calculus inequalities. 
We consider the initial boundary value problem \eqref{HS} in Sections \ref{sect:EstIV1}--\ref{sect:Proof1}, which are the first part of this paper. 
In Section \ref{sect:EstIV1} we evaluate initial values for time derivatives of $\bm{u}$ in terms of the initial data 
$(\bm{u}_0^\mathrm{in},\bm{u}_1^\mathrm{in})$ and the forcing term $\bm{f}$ and state precisely the compatibility conditions on the data. 
In Section \ref{sect:EE} we derive a basic energy estimate in Proposition \ref{prop:BEE} and a higher order energy estimate in Proposition \ref{prop:HOEE} 
for the solution to the regularized problem \eqref{rLS} including the case $\varepsilon=0$. 
In Section \ref{sect:Proof1} we prove Theorem \ref{Th1}. 
To this end, we first show the well-posedness of the regularized problem in the case $Q(s,t)=O$ and $\varepsilon>0$ 
by transforming the problem on the interval $(0,1)$ into a problem for a non-degenerate hyperbolic system on the unit disc $D$. 
We also derive an additional boundary regularity of the solution. 
We then show the well-posedness of the problem with a non-zero localized term $Q(s,t)\bm{u}'(1,t)$ in the case $\varepsilon>0$ by the standard Picard iteration. 
Thanks to the energy estimate obtained in Section \ref{sect:EE} we can pass to the limit $\varepsilon\to+0$ 
and obtain a solution $\bm{u}$ of the original problem \eqref{HS}. 
We then consider the initial boundary value problem \eqref{LEq} and \eqref{LBVP} in Sections \ref{sect:TBVP}--\ref{sect:Proof2}, 
which are the second part of this paper. 
In Section \ref{sect:TBVP} we analyze two-point boundary value problems related to \eqref{LBVP} and \eqref{defphi}, especially, 
derive estimates for the solution $\phi$ of \eqref{defphi} and those for the lower order part $\nu_\mathrm{l}$ of $\nu$ in terms of time dependent norms. 
These estimates guarantee that the term $(\nu_\mathrm{l}\bm{x}')'$ in \eqref{rLEq} is of lower order. 
In Section \ref{sect:EstIV2} we evaluate initial values for time derivatives of $(\bm{y},\nu)$ in terms of the initial data 
$(\bm{y}_0^\mathrm{in},\bm{y}_1^\mathrm{in})$ and the forcing terms $(\bm{f},h)$ and state precisely the compatibility conditions on the data. 
In Section \ref{sect:Proof2} we prove Theorem \ref{Th2}. 
To show the existence of the solution, we use the method of successive approximation. 
In each steps, we apply Theorem \ref{Th1}.

\medskip
\noindent
{\bf Notation}. \ 
For $1\leq p\leq\infty$, we denote by $L^p$ the Lebesgue space on the open interval $(0,1)$. 
For non-negative integer $m$, we denote by $H^m$ the $L^2$ Sobolev space of order $m$ on $(0,1)$. 
The norm of a Banach space $B$ is denoted by $\|\cdot\|_B$. 
The inner product in $L^2$ is denoted by $(\cdot,\cdot)_{L^2}$. 
We put $\dt=\frac{\partial}{\partial t}$ and $\ds=\frac{\partial}{\partial s}$. 
The norm of a weighted $L^p$ space with a weight $s^\alpha$ is denoted by $\|s^\alpha u\|_{L^p}$, so that 
$\|s^\alpha u\|_{L^p}^p=\int_0^1s^{\alpha p}|u(s)|^p \mathrm{d}s$ for $1\leq p<\infty$. 
It is sometimes denoted by $\|\sigma^\alpha u\|_{L^p}$, too. 
This would cause no confusion. 
$[P,Q]=PQ-QP$ denotes the commutator. 
We denote by $C(a_1, a_2, \ldots)$ a positive constant depending on $a_1, a_2, \ldots$. 
$f\lesssim g$ means that there exists a non-essential positive constant C such that $f\leq Cg$ holds. 
$f\simeq g$ means that $f\lesssim g$ and $g\lesssim f$ hold. 
$a_1 \vee a_2 = \max\{a_1,a_2\}$.

\medskip
\noindent
{\bf Acknowledgement} \\
T. I. is partially supported by JSPS KAKENHI Grant Number JP22H01133.

\section{Main results}\label{sect:results}
In order to state our main results, we first introduce function spaces that we are going to use in this paper. 
For a non-negative integer $m$, following Reeken \cite{Reeken1979-1, Reeken1979-2}, Takayama \cite{Takayama2018}, and Iguchi and Takayama \cite{IguchiTakayama2023}, 
we define a weighted Sobolev space $X^m$ as a set of all function $u=u(s)\in L^2$ equipped with a norm $\|\cdot\|_{X^m}$ defined by 
\begin{equation}\label{WSS}
\|u\|_{X^m}^2 =
\begin{cases}
 \displaystyle
 \|u\|_{H^k}^2 + \sum_{j=1}^k\|s^j\ds^{k+j}u\|_{L^2}^2 &\mbox{for}\quad m=2k, \\
 \displaystyle
 \|u\|_{H^k}^2 + \sum_{j=1}^{k+1}\|s^{j-\frac12}\ds^{k+j}u\|_{L^2}^2 &\mbox{for}\quad m=2k+1.
\end{cases}
\end{equation}
For a function $u=u(s,t)$ depending also on time $t$ and for integers $m$ and $l$ satisfying $0\leq l\leq m$, 
we introduce a norm $\opnorm{\cdot}_{m,l}$ and the space $\mathscr{X}_T^{m,l}$ by 
\[
\opnorm{u(t)}_{m,l}^2 = \sum_{j=0}^l \|\dt^j u(t)\|_{X^{m-j}}^2, \qquad
\mathscr{X}_T^{m,l} = \bigcap_{j=0}^l C^j([0,T],X^{m-j}), 
\]
and put $\opnorm{\cdot}_m=\opnorm{\cdot}_{m,m}$, $\opnorm{\cdot}_{m,*}=\opnorm{\cdot}_{m,m-1}$, 
$\mathscr{X}_T^m = \mathscr{X}_T^{m,m}$, and $\mathscr{X}_T^{m,*} = \mathscr{X}_T^{m,m-1}$. 
We use a notational convention $\opnorm{\cdot}_{0,*}=0$.

For a non-negative integer $m$, we define another weighted Sobolev space $Y^m$ as the set of all function $u=u(s)$ defined in the open interval $(0,1)$ 
equipped with a norm $\|\cdot\|_{Y^m}$ defined by 
\[
\|u\|_{Y^m}^2 =
\begin{cases}
 \|s^\frac12 u\|_{L^2}^2 &\mbox{for}\quad m=0, \\
 \displaystyle
 \|u\|_{H^k}^2 + \sum_{j=1}^{k+1}\|s^j\ds^{k+j}u\|_{L^2}^2 &\mbox{for}\quad m=2k+1, \\
 \displaystyle
 \|u\|_{H^k}^2 + \sum_{j=1}^{k+2}\|s^{j-\frac12}\ds^{k+j}u\|_{L^2}^2 &\mbox{for}\quad m=2k+2.
\end{cases}
\]
This norm is introduced so that $\|u\|_{X^{m+1}}^2=\|u\|_{L^2}^2+\|u'\|_{Y^m}^2$ holds for $m=0,1,2,\ldots$. 
For a function $u=u(s,t)$ depending also on time $t$ and 
for a non-negative integer $m$, 
we introduce a norm 
$\opnorm{ \cdot }_{m}^\dag$ and the space $\mathscr{Y}_T^{m}$ by 
\[
\opnorm{ u(t) }_{m}^\dag = \sum_{j=0}^m \|\dt^j u(t)\|_{Y^{m-j}}, \qquad
\mathscr{Y}_T^{m} = \bigcap_{j=0}^m C^j([0,T];Y^{m-j}).
\]
We use a notational convention $\opnorm{\cdot}_{-1}^\dag=0$.

For a function $u=u(t)$ of time $t$, following Iguchi and Lannes \cite{IguchiLannes2021}, 
we use weighted norms with an exponential function $\mathrm{e}^{-\gamma t}$ for $\gamma>0$ defined by 
\[
|u|_{L_\gamma^p(0,t)} = \left( \int_0^t \mathrm{e}^{-p\gamma t'}|u(t')|^p{\rm d}t' \right)^\frac1p, \qquad
  |u|_{H_\gamma^m(0,t)} = \Biggl( \sum_{j=0}^m |\dt^ju|_{L_\gamma^2(0,t)}^2 \Biggr)^\frac12,
\]
and put 
\[
I_{\gamma,t}(u)=\sup_{0\leq t'\leq t}\mathrm{e}^{-\gamma t'}|u(t')| + \sqrt{\gamma}|u|_{L_\gamma^2(0,t)}.
\]
We denote by $S_{\gamma,t}^*(\cdot)$ its dual norm for the $L_\gamma^2(0,t)$ scalar product, that is, 
\begin{equation}\label{defSstar}
S_{\gamma,t}^*(u) = \sup_{\varphi} \biggl\{ \biggl| \int_0^t \mathrm{e}^{-2\gamma t'}u(t')\varphi(t'){\rm d}t' \biggr|
 \,;\, I_{\gamma,t}(\varphi) \leq 1 \biggr\}.
\end{equation}
From this definition, we get directly the following upper bounds 
\begin{equation}\label{propSstar}
S^*_{\gamma,t}(u)\leq |u|_{L^1_{\gamma}(0,t)}
\quad\mbox{ and }\quad
S^*_{\gamma,t}(u)\leq \frac{1}{\sqrt{\gamma}}|u|_{L^2_{\gamma}(0,t)} \leq \frac{1}{\gamma}I_{\gamma,t}(u).
\end{equation}

In order to state our result on the well-posedness of the problem \eqref{LS}, we need to impose precise assumptions on the coefficient matrices $A$ and $Q$.

\begin{assumption}\label{ass:BEE}
Let $M_0$ and $M_1$ be positive constants. 
For any $(s,t)\in(0,1)\times(0,T)$, $A(s,t)$ is symmetric and it holds that 
\[
\begin{cases}
 M_0^{-1}s\mathrm{Id} \leq A(s,t) \leq M_0s\mathrm{Id}, \\
 |A'(s,t)|+s^\frac12|Q(s,t)|+\|Q(t)\|_{L^2} \leq M_0, \\
 |\dt A'(s,t)|+\opnorm{Q(t)}_1 \leq M_1.
\end{cases}
\]
\end{assumption}

These assumptions guarantee a basic energy estimate for the solution of the problem \eqref{LS} 
and the following assumptions guarantee higher order energy estimates together with the existence of the solution.

\begin{assumption}\label{ass:HOEE}
Let $m\geq2$ be an integer, $T$, $M_0$, and $M_1$ be positive constants. 
 \setlength{\parskip}{-2mm}
\begin{enumerate}
 \setlength{\itemsep}{-0.5mm}
\item[{\rm (i)}]
$A' \in \mathscr{X}_T^{m-2}\cap\mathscr{X}_T^{2,*}$ and $Q\in\mathscr{X}_T^{m-2\vee1}$. 
\item[{\rm (ii)}]
$\dt^{m-1}A', \dt^{m-1\vee2}Q\in L^\infty(0,T;L^2)$ and $\dt^2A'\in L^\infty(0,T;X^1)$. 
\item[{\rm (iii)}]
In the case $m\geq3$, for any $t\in(0,T)$ it holds that 
\[
\begin{cases}
 \opnorm{A'(t)}_{m-2}+\opnorm{A'(t)}_{2,*}+\opnorm{Q(t)}_{m-2} \leq M_0, \\
 \|\dt^{m-1}A'(t)\|_{L^2}+\|\dt^{m-1}Q(t)\|_{L^2}+\|\dt^2A'(t)\|_{X^1} \leq M_1.
\end{cases}
\]
\end{enumerate}
\end{assumption}

The following theorem is one of main results in this paper and gives a well-posedness of the problem \eqref{LS} in the weighted Sobolev space $X^m$.

\begin{theorem}\label{Th1}
Let $m\geq2$ be an integer, $T>0$, and assume that Assumptions \ref{ass:BEE} and \ref{ass:HOEE} are satisfied with positive constants $M_0$ and $M_1$. 
Then, for any data $\bm{u}_0^\mathrm{in}\in X^m$, $\bm{u}_1^\mathrm{in}\in X^{m-1}$, and $\bm{f}\in\mathscr{X}_T^{m-2}$ satisfying 
$\dt^{m-1}\bm{f}\in L^1(0,T;L^2)$ and the compatibility conditions up to order $m-1$ in the sense of Definition \ref{def:CC1} below, 
there exists a unique solution $\bm{u}\in\mathscr{X}_T^m$ to the initial boundary value problem \eqref{LS}. 
Moreover, the solution satisfies the estimate 
\begin{equation}\label{EE1}
I_{\gamma,t}(\opnorm{ \bm{u}(\cdot) }_m)
\leq C_0\left\{ \|\bm{u}_0^\mathrm{in}\|_{X^m} + \|\bm{u}_1^\mathrm{in}\|_{X^{m-1}} + I_{\gamma,t}(\opnorm{\bm{f}(\cdot)}_{m-2})
 +S_{\gamma,t}^*(\|\dt^{m-1}\bm{f}(\cdot)\|_{L^2}) \right\}
\end{equation}
for any $t\in[0,T]$ and any $\gamma\geq\gamma_1$, where $C_0>0$ depends only on $m$ and $M_0$ and $\gamma_1>0$ depends also on $M_1$. 
\end{theorem}

\begin{remark}\label{re:1}
In view of \eqref{propSstar} we see that the solution obtain in Theorem \ref{Th1} satisfies 
\[
\opnorm{ \bm{u}(t) }_m
\leq C_0\mathrm{e}^{C_1t}\left( \|\bm{u}_0^\mathrm{in}\|_{X^m} + \|\bm{u}_1^\mathrm{in}\|_{X^{m-1}}
 + \sup_{0\leq t'\leq t}\opnorm{ \bm{f}(t') }_{m-2} + \int_0^t\|\dt^{m-1}\bm{f}(t')\|_{L^2}\mathrm{d}t' \right)
\]
for any $t\in[0,T]$, where $C_1>0$ depends only on $m$, $M_0$, and $M_1$. 
\end{remark}

We proceed to consider the linearized system \eqref{LEq} and \eqref{LBVP}. 
In order to state our result on the well-posedness of the problem \eqref{LEq} and \eqref{LBVP}, 
we need to impose precise assumptions on $\bm{x}$ and $\tau$. 
We recall that the coefficient matrices $A$ and $Q$ in the linearized problem \eqref{rLEq} are given by 
$A(s,t)=\tau(s,t)\mathrm{Id}$ and $Q(s,t)=-(\phi\bm{x}')'(s,t)\otimes(\bm{g}+2\tau\bm{x}'')(1,t)$. 
In order that these coefficient matrices satisfy Assumptions \ref{ass:BEE} and \ref{ass:HOEE}, we impose the following assumptions.

\begin{assumption}\label{ass:xtau}
Let $m\geq2$ be an integer, $T$, $M_0$, and $M_1$ be positive constants. 
 \setlength{\parskip}{-2mm}
\begin{enumerate}
 \setlength{\itemsep}{-0.5mm}
\item[{\rm (i)}]
For any $(s,t)\in(0,1)\times(0,T)$, it holds that 
\[
\begin{cases}
 M_0^{-1}s \leq \tau(s,t) \leq M_0s, \quad \bm{x}(1,t)=\bm{0}, \\
 \sup_{0\leq t\leq T}\bigl( \opnorm{ (\tau',\dot{\tau}')(t) }_{2,*} + \opnorm{ (\bm{x},\dot{\bm{x}})(t) }_4 \bigr) < \infty.
\end{cases}
\]
\item[{\rm (ii)}]
In the case $m=2$, for any $t\in(0,T)$ it holds that 
\[
\begin{cases}
 \|\tau'(t)\|_{L^\infty} + \|\bm{x}(t)\|_{X^3} \leq M_0, \\
 \|\dot{\tau}'(t)\|_{L^\infty}+\|\bm{x}(t)\|_{X^4}+\|\dot{\bm{x}}(t)\|_{X^3} \leq M_1.
\end{cases}
\]
\item[{\rm (iii)}]
In the case $m=3$, for any $t\in(0,T)$ it holds that 
\[
\begin{cases}
 \opnorm{ \tau'(t) }_{2,*} + \opnorm{ \bm{x}(t) }_4 \leq M_0, \\
 \opnorm{ \dot{\tau}'(t) }_{2,*} + \opnorm{ \dot{\bm{x}}(t) }_4 \leq M_1.
\end{cases}
\]
\item[{\rm (iv)}]
In the case $m\geq4$, for any $t\in(0,T)$ it holds that 
\[
\begin{cases}
 \opnorm{ \tau'(t) }_{m-2} + \opnorm{ (\bm{x},\dot{\bm{x}})(t) }_m \leq M_0, \\
 \opnorm{ \dot{\tau}'(t) }_{m-2} + \opnorm{ \ddot{\bm{x}}(t) }_m \leq M_1.
\end{cases}
\]
\end{enumerate}
\end{assumption}

In order to guarantee that the term $(\nu_\mathrm{l}\bm{x}')'$ in \eqref{rLEq} is of lower order, 
in addition to Assumption \ref{ass:xtau}, we impose the following assumptions.

\begin{assumption}\label{ass:addxtau}
Let $m\geq2$ be an integer, $T$, $M_0$, and $M_1$ be positive constants. 
 \setlength{\parskip}{-2mm}
\begin{enumerate}
 \setlength{\itemsep}{-0.5mm}
\item[{\rm (i)}]
In the case $m=2$, for any $t\in(0,T)$ it holds that 
\[
\begin{cases}
 \|\bm{x}'(t)\|_{L^\infty}+\|\dot{\bm{x}}(t)\|_{X^2} \leq M_0, \\
 \|\dot{\bm{x}}'(t)\|_{L^\infty}+\|\ddot{\bm{x}}(t)\|_{X^2} \leq M_1.
\end{cases}
\]
\item[{\rm (ii)}]
In the case $m=3$, $\bm{x}\in C^1([0,T];X^4)$ and $\|\dot{\bm{x}}(t)\|_{X^4}\leq M_0$ for $0\leq t\leq T$. 
\end{enumerate}
\end{assumption}

In order to obtain an optimal regularity of $\nu$ relative to $\bm{y}$, in addition to Assumptions \ref{ass:xtau} and \ref{ass:addxtau}, 
we impose the following assumptions.

\begin{assumption}\label{ass:addxtau2}
Let $T$ and $M_0$ be positive constants. 
In the case $m=2$, $\bm{x}\in C^1([0,T];X^4)$ and $\|(\bm{x},\dot{\bm{x}})(t)\|_{X^4}\leq M_0$ for $0\leq t\leq T$. 
\end{assumption}

The following theorem is another main result in this paper and gives a well-posedness of the problem \eqref{LEq} and \eqref{LBVP} in the weighted Sobolev space $X^m$.

\begin{theorem}\label{Th2}
Let $m\geq2$ be an integer, $T>0$, and assume that Assumptions \ref{ass:xtau} and \ref{ass:addxtau} are satisfied with positive constants $M_0$ and $M_1$. 
Suppose that the data $\bm{y}_0^\mathrm{in}\in X^m$, $\bm{y}_1^\mathrm{in}\in X^{m-1}$, $\bm{f}\in\mathscr{X}_T^{m-2}$, and $h$ 
satisfy $\dt^{m-1}\bm{f}\in L^1(0,T;L^2)$, $s^\frac12\dt^{m-2}h \in C^0([0,T];L^1)$, $s^\frac12\dt^{m-1}h \in L^1((0,1)\times(0,T))$. 
In the case $m\geq 3$, assume also that $h\in\mathscr{Y}_T^{m-3}$. 
In addition, suppose that the data satisfy the compatibility conditions up to order $m-1$ in the sense of Definition \ref{def:CC2} below. 
Then, there exists a unique solution $(\bm{y},\nu)$ to the problem \eqref{LEq} and \eqref{LBVP} in the class $\bm{y}\in\mathscr{X}_T^m$ and $\nu'\in\mathscr{X}_T^{m-2}$. 
Moreover, the solution satisfies the estimate 
\begin{align}\label{EstLP}
I_{\gamma,t}(\opnorm{ \bm{y}(\cdot) }_m + \opnorm{ \nu'(\cdot) }_{m-2})
&\leq C_0\bigl\{ \|\bm{y}_0^\mathrm{in}\|_{X^m} + \|\bm{y}_1^\mathrm{in}\|_{X^{m-1}} \\
&\qquad
 +I_{\gamma,t}( \opnorm{ \bm{f}(\cdot) }_{m-2} + \opnorm{ h(\cdot) }_{m-3}^\dag + \|s^\frac12\dt^{m-2}h(\cdot)\|_{L^1} ) \nonumber \\
&\qquad
 +S_{\gamma,t}^*(\|\dt^{m-1}\bm{f}(\cdot)\|_{L^2}) \bigr\} + C_1S_{\gamma,t}^*(\|s^\frac12\dt^{m-1}h(\cdot)\|_{L^1}) \nonumber
\end{align}
for any $t\in[0,T]$ and any $\gamma\geq\gamma_1$, where $C_0>0$ depends only on $m$ and $M_0$ and $C_1, \gamma_1>0$ depend also on $M_1$. 
Furthermore, if we assume additionally Assumption \ref{ass:addxtau2} and $h\in\mathscr{Y}_T^{m-2}$, then we have $\nu'\in\mathscr{X}_T^{m-1,*}$ and 
\begin{equation}\label{EstNu}
\opnorm{\nu'(t)}_{m-1,*} \leq C_0\bigl( \opnorm{ \bm{y}(t) }_m + \opnorm{ \nu'(t) }_{m-2} + \opnorm{h(t)}_{m-2}^\dag \bigr)
\end{equation}
for any $t\in[0,T]$. 
\end{theorem}

\begin{remark}\label{re:exist}
Although the map $(\bm{x},\tau)\mapsto(\bm{y},\nu)$ reveals loss of twice derivatives, by a standard procedure of a quasilinearization 
we can construct a unique solution $(\bm{x},\tau)$ to the nonlinear problem \eqref{HS} and \eqref{BVP} in the case $m\geq6$. 
However, a priori estimates for the solution were obtained in the case $m\geq4$ by Iguchi and Takayama \cite{IguchiTakayama2023}, 
so that it is natural to expect that the well-posedness of the problem holds also in the case $m=4,5$. 
In order to show this, we need detailed analysis on compatibility conditions to the initial data, 
which do not have any standard form due to a nonlocal property caused by the tension $\tau$. 
Therefore, we postpone this well-posedness part to the nonlinear problem \eqref{HS} and \eqref{BVP} in our future work. 
\end{remark}

\section{Basic properties of the weighted Sobolev spaces}\label{sect:pre}
In this preliminary section, we present basic properties of the weighted Sobolev spaces $X^m$ and $Y^m$ and related calculus inequalities. 
Many of them are proved in Takayama \cite{Takayama2018} and Iguchi and Takayama \cite{IguchiTakayama2023}. 
Let $D$ be the unit disc in $\mathbb{R}^2$ and $H^m(D)$ the $L^2$ Sobolev space of order $m$ on $D$. 
For a function $u$ defined in the open interval $(0,1)$, we define $u^\sharp(x_1,x_2)=u(x_1^2+x_2^2)$ which is a function on $D$.

\begin{lemma}[{\cite[Proposition 3.2]{Takayama2018}}]\label{lem:NormEq}
Let $m$ be a non-negative integer. 
The map $X^m\ni u \mapsto u^\sharp \in H^m(D)$ is bijective and it holds that $\|u\|_{X^m} \simeq \|u^\sharp\|_{H^m(D)}$ 
for any $u\in X^m$. 
\end{lemma}

\begin{lemma}[{\cite[Lemma 4.3]{IguchiTakayama2023}}]\label{lem:embedding}
For any $\epsilon>0$ there exists a positive constant $C_\epsilon=C(\epsilon)$ such that 
for any $u\in X^1$ we have $\|s^\epsilon u\|_{L^\infty} \leq C_\epsilon\|u\|_{X^1}$. 
\end{lemma}

\begin{lemma}[{\cite[Lemma 4.5]{IguchiTakayama2023}}]\label{lem:embedding2}
For a non-negative integer $m$, we have $\|su'\|_{X^m} \leq \|u\|_{X^{m+1}}$, $\|u'\|_{X^m} \leq \|u\|_{X^{m+2}}$, 
and $\|\ds^mu\|_{L^\infty} \lesssim \|u\|_{X^{2m+2}}$. 
\end{lemma}

\begin{lemma}[{\cite[Lemma 4.6]{IguchiTakayama2023}}]\label{lem:CalIneqLp1}
For a positive integer $k$ there exists a positive constant $C$ such that for any $p\in[2,\infty]$ we have 
\[
\begin{cases}
 \|s^{j-\frac12-\frac1p}\ds^{k+j-1}u\|_{L^p} \leq C\|u\|_{X^{2k}} &\mbox{for}\quad j=1,2,\ldots,k, \\
 \|s^{j-\frac1p}\ds^{k+j}u\|_{L^p} \leq C\|u\|_{X^{2k+1}} &\mbox{for}\quad j=1,2,\ldots,k.
\end{cases}
\]
\end{lemma}

\begin{lemma}[{\cite[Lemma 4.7]{IguchiTakayama2023}}]\label{lem:Algebra}
For a non-negative integer $m$, we have $\|uv\|_{L^2}\lesssim\|u\|_{X^1}\|v\|_{X^1}$ and $\|uv\|_{X^m} \lesssim\|u\|_{X^{m \vee 2}} \|v\|_{X^m}$. 
\end{lemma}

\begin{lemma}[{\cite[Lemma 4.8]{IguchiTakayama2023}}]\label{lem:EstCompFunc1}
Let $m$ be an non-negative integer, $\Omega$ an open set in $\mathbb{R}^N$, and $F\in C^m(\Omega)$. 
There exists a positive constant $C=C(m,N)$ such that if $u\in X^m$ takes its value in a compact set $K$ in $\Omega$, then we have 
$\|F(u)\|_{X^m} \leq C\|F\|_{C^m(K)}(1+\|u\|_{X^m})^m$. 
If, in addition, $u$ depends also on time $t$, then we have also 
\[
\begin{cases}
 \opnorm{ F(u(t)) }_m \leq C\|F\|_{C^m(K)}(1+\opnorm{ u(t) }_m )^m, \\
 \opnorm{ F(u(t)) }_{m,*} \leq C\|F\|_{C^m(K)}(1+\opnorm{ u(t) }_{m,*} )^m.
\end{cases}
\]
\end{lemma}

\begin{lemma}[{\cite[Lemma 4.9]{IguchiTakayama2023}}]\label{lem:commutator}
Let $j$ be a non-negative integer. 
It holds that 
\begin{align*}
& \|s^\frac{j}{2}[\ds^{j+1},u]v\|_{L^2} \lesssim
\begin{cases}
 \min\{ \|u'\|_{L^\infty}\|v\|_{L^2}, \|u'\|_{L^2}\|v\|_{L^\infty} \} &\mbox{for}\quad j=0, \\
 \min\{ \|u'\|_{X^2}\|v\|_{X^1}, \|u'\|_{X^1}\|v\|_{X^2} \} &\mbox{for}\quad j=1, \\
 \|u'\|_{X^j}\|v\|_{X^j} &\mbox{for}\quad j\geq2.
\end{cases}
\end{align*}
\end{lemma}

\begin{lemma}[{\cite[Lemma 9.1]{IguchiTakayama2023}}]\label{lem:EstAu}
If $a|_{s=0}$, then we have 
\[
\|(au')'\|_{X^m} \lesssim
\begin{cases}
 \min\{ \|a'\|_{L^\infty} \|u\|_{X^2}, \|a'\|_{X^1} \|u\|_{X^3} \} &\mbox{for}\quad m=0, \\
 \min\{ \|a'\|_{X^{m \vee 2}} \|u\|_{X^{m+2}}, \|a'\|_{X^m} \|u\|_{X^{m+2 \vee 4}} \} &\mbox{for}\quad m=0,1,2,\ldots.
\end{cases}
\]
\end{lemma}

\begin{lemma}\label{lem:EstAu2}
Let $m$ and $j$ be integers such that $1\leq j\leq m$. 
If $a|_{s=0}$, then we have 
\[
\opnorm{ (au')'(t) }_{m,j} \lesssim
\begin{cases}
 \min\{ \opnorm{a'(t)}_1 \opnorm{u(t)}_{4,1}, \opnorm{a'(t)}_{2,1} \opnorm{u(t)}_{3,1}  \} &\mbox{for}\quad m=1, \\
 \opnorm{a'(t)}_{m,j} \opnorm{u(t)}_{m+2,j} &\mbox{for}\quad m=2,3,\ldots.
\end{cases}
\]
\end{lemma}

\begin{proof}
By Lemma \ref{lem:EstAu}, we see that 
\begin{align*}
\opnorm{ (au')' }_1
&\leq \|(au')'\|_{X^1} + \|(a\dot{u}')'\|_{L^2} + \|(\dot{a}u')'\|_{L^2} \\
&\lesssim \|a'\|_{X^1}\|u\|_{X^4} + \|a'\|_{X^1}\|\dot{u}\|_{X^3} + \|\dot{a}'\|_{L^2}\|u\|_{X^4} \\
&\lesssim \opnorm{a'}_1 \opnorm{u}_{4,1}
\end{align*}
and that 
\begin{align*}
\opnorm{ (au')' }_1
&\lesssim \|a'\|_{X^2}\|u\|_{X^3} + \|a'\|_{X^2}\|\dot{u}\|_{X^2} + \|\dot{a}'\|_{X^1}\|u\|_{X^3} \\
&\lesssim \opnorm{a'}_{2,1} \opnorm{u}_{3,1}.
\end{align*}
These imply the first estimate of the lemma. 
We then consider the case $m\geq2$.
\[
\opnorm{ (au')' }_{m,j} \lesssim \sum_{0\leq k\leq j}\sum_{k_1+k_2=k}\|((\dt^{k_1}a)(\dt^{k_2}u)')'\|_{X^{m-k}}.
\]
We evaluate $I(k_1,k_2;k)=\|((\dt^{k_1}a)(\dt^{k_2}u)')'\|_{X^{m-k}}$, where $0\leq k\leq j$ and $k_1+k_2=k$. 
In the following calculations, we use Lemma \ref{lem:EstAu}. 
\begin{enumerate}
\item[(i)]
The case $k\leq m-2$. 
\begin{align*}
I(k_1,k_2;k)
&\lesssim \|\dt^{k_1}a'\|_{X^{m-k}} \|\dt^{k_2}u\|_{X^{m+2-k}} \\
&\lesssim \opnorm{a'}_{m,k} \opnorm{u}_{m+2,k}.
\end{align*}
\item[(ii)]
The case $j=m-1$. 
\begin{align*}
I(k_1,k_2;k)
&\lesssim 
 \begin{cases}
  \|\dt^{j}a'\|_{X^1} \|u\|_{X^4} &\mbox{for}\quad k_1=j, \\
  \|\dt^{k_1}a'\|_{X^2} \|\dt^{k_2}u\|_{X^3} &\mbox{for}\quad k_1\leq j-1
 \end{cases} \\
&\lesssim \opnorm{a'}_{m,j} \opnorm{u}_{m+2,j}.
\end{align*}
\item[(iii)]
The case $j=m$. 
\begin{align*}
I(k_1,k_2;k)
&\lesssim 
 \begin{cases}
  \|\dt^ja'\|_{L^2} \|u\|_{X^4} &\mbox{for}\quad k_1=j, \\
  \|\dt^{j-1}a'\|_{X^1} \|\dt u\|_{X^3} &\mbox{for}\quad k_1=j-1 \\
  \|\dt^{k_1}a'\|_{X^2} \|\dt^{k_2}u\|_{X^2} &\mbox{for}\quad k_1\leq j-2
 \end{cases} \\
&\lesssim \opnorm{a'}_{m,j} \opnorm{u}_{m+2,j}.
\end{align*}
\end{enumerate}
These imply the second estimate of the lemma. 
\end{proof}

As in Iguchi and Takayama \cite{IguchiTakayama2023}, we will use an averaging operator $\mathscr{M}$ defined by 
\begin{equation}\label{defM}
(\mathscr{M}u)(s) = \frac{1}{s}\int_0^su(\sigma) \mathrm{d}\sigma. 
\end{equation}

\begin{lemma}[{\cite[Corollary 4.13]{IguchiTakayama2023}}]\label{lem:WEM1}
Let $j$ be non-negative integer, $1\leq p\leq \infty$, and $\beta<j+1-\frac1p$. 
Then, we have 
\[
\|s^\beta\ds^j(\mathscr{M}u)\|_{L^p} \leq \frac{1}{j+1-\beta-\frac1p}\|s^\beta\ds^j u\|_{L^p}.
\]
Particularly, $\|\mathscr{M}u\|_{X^m} \leq 2\|u\|_{X^m}$ for $m=0,1,2,\ldots$. 
\end{lemma}

\begin{lemma}\label{lem:Comm}
Let $m$ be a positive integer and assume that $a|_{s=0}=0$. 
Then, we have 
 \setlength{\parskip}{0mm}
\begin{enumerate}
 \setlength{\itemsep}{-0.5mm}
\item[{\rm (i)}]
$\|\dt([\dt^m,a]u')'\|_{L^2} \lesssim \bigl( \|\dt a'\|_{L^\infty}+\|\dt^2a'\|_{X^1}+\sum_{j=3}^{m+1}\|\dt^ja'\|_{L^2} \bigr)\opnorm{u}_{m+2,m}$. 
\item[{\rm (ii)}]
$\|([\dt^m,a]u')'\|_{X^1} \lesssim \|a'\|_{m+1,*}(\opnorm{u}_{m+2,m-1} + \|u'\|_{X^2})$. 
\item[{\rm (iii)}]
$\|\dt[\dt^m,q]u'(1,t)\|_{L^2} \lesssim \bigl( \sum_{j=1}^{m+1}\|\dt^jq\|_{L^2} \bigr)\opnorm{u}_{m+2,m}$. 
\end{enumerate}
\end{lemma}

\begin{proof}
By Lemma \ref{lem:EstAu}, we see that 
\begin{align*}
&\|\dt([\dt^m,a]u')'\|_{L^2} \\
&\lesssim \|((\dt a)(\dt^m u)')'\|_{L^2} + \|((\dt^2 a)(\dt^{m-1}u)')'\|_{L^2} + \sum_{j=3}^{m+1}\|((\dt^ja)(\dt^{m+1-j}u)')'\|_{L^2} \\
&\lesssim \|\dt a'\|_{L^\infty}\|\dt^m u\|_{X^2} + \|\dt^2 a'\|_{X^1}\|\dt^{m-1} u\|_{X^3} + \sum_{j=3}^{m+1}\|\dt^ja'\|_{L^2} \|\dt^{m+1-j}u\|_{X^4},
\end{align*}
which implies (i). 

Similarly, we have 
\begin{align*}
\|([\dt^m,a]u')'\|_{X^1} 
&\lesssim \sum_{j=1}^{m-1}\|((\dt^ja)(\dt^{m-j}u)')'\|_{X^1} + \|((\dt^ma)u')'\|_{X^1}.
\end{align*}
Here, for $1\leq j\leq m-1$ we see that 
\begin{align*}
\|((\dt^ja)(\dt^{m-j}u)')'\|_{X^1}
&\lesssim \|\dt^ja'\|_{X^2}\|\dt^{m-j}u\|_{X^3} 
 \lesssim \opnorm{a'}_{m+1,*} \opnorm{u}_{m+2,m-1}.
\end{align*}
As for the second term, in the case $m\geq2$ we have, 
\begin{align*}
\|((\dt^ma)u')'\|_{X^1} 
&\lesssim \|\dt^ma'\|_{X^1}\|u\|_{X^4} 
 \lesssim \opnorm{a'}_{m+1,*} \opnorm{u}_{m+2,m-1},
\end{align*}
while in the case $m=1$ we evaluate it as 
\begin{align*}
\|((\dt a)u')'\|_{X^1}
&\leq \|s^\frac12((\dt a)u')''\|_{L^2} + \|((\dt a)u')'\|_{L^2} \\
&\leq \|s^{-\frac12}\dt a\|_{L^\infty}\|su'''\|_{L^2} + (\|s^\frac12\dt a'\|_{L^\infty} + \|\dt a\|_{L^\infty})\|u''\|_{L^2} \\
&\quad\;
 + \|s^\frac12\dt a''\|_{L^2}\|u'\|_{L^\infty} + \|\dt a'\|_{L^2}\|u'\|_{L^\infty}.
\end{align*}
Since $\dt a|_{s=0}=0$, we have $\dt a=s\mathscr{M}(\dt a')$ with the averaging operator $\mathscr{M}$ defined by \eqref{defM}. 
Therefore, by Lemmas \ref{lem:WEM1} and \ref{lem:embedding} we obtain 
$\|s^{-\frac12}\dt a\|_{L^\infty} = \|s^\frac12\mathscr{M}(\dt a')\|_{L^\infty} \leq 2\|s^\frac12\dt a'\|_{L^\infty} \lesssim \|\dt a'\|_{X^1}$, 
so that 
\begin{equation}\label{AuxEst1}
\|((\dt a)u')'\|_{X^1} \lesssim \|\dt a'\|_{X^1}\|u'\|_{X^2}. 
\end{equation}
These estimates imply (ii). 

In view of $|\dt^ju'(1,t)| \lesssim \|\dt^ju\|_{X^2} \leq \opnorm{u}_{m+2,m}$ for $0\leq j\leq m$, 
the proof of (iii) is straightforward. 
\end{proof}

\begin{lemma}[{\cite[Lemma 4.10]{IguchiTakayama2023}}]\label{lem:CalIneq1}
It holds that 
\[
\|u'v'\|_{Y^m} \lesssim
\begin{cases}
 \|u\|_{X^{m+2}} \|v\|_{X^{m+2}} &\mbox{for}\quad m=0,1, \\
 \|u\|_{X^{m+1\vee4}} \|v\|_{X^{m+1}} &\mbox{for}\quad m=0,1,2,\ldots.
\end{cases}
\]
\end{lemma}

\begin{lemma}\label{lem:CalIneqY1}
For a positive integer $m$, we have 
\[
\opnorm{(u'v')(t)}_m^\dag \lesssim 
\begin{cases}
 \opnorm{u(t)}_{3,1} \opnorm{v(t)}_{3,1} &\mbox{for}\quad m=1, \\
 \opnorm{u(t)}_{m+1\vee4,m} \opnorm{v(t)}_{m+1,m} &\mbox{for}\quad m\geq1.
\end{cases}
\]
\end{lemma}

\begin{proof}
By Lemma \ref{lem:CalIneq1}, we see that 
\begin{align*}
\opnorm{(u'v')(t)}_1^\dag
&\leq \|u'v'\|_{Y^1} + \|(\dt u)'v'\|_{Y^0} + \|u'(\dt v)'\|_{Y^0} \\
&\lesssim \|u\|_{X^3}\|v\|_{X^3} + \|\dt u\|_{X^2}\|v\|_{X^2} + \|u\|_{X^2}\|\dt v\|_{X^2} \\
&\lesssim \opnorm{u}_{3,1} \opnorm{v}_{3,1}
\end{align*}
and that 
\begin{align*}
\opnorm{(u'v')(t)}_1^\dag
&\lesssim \|u\|_{X^4}\|v\|_{X^2} + \|\dt u\|_{X^2}\|v\|_{X^2} + \|u\|_{X^4}\|\dt v\|_{X^1} \\
&\lesssim \opnorm{u}_{4,1} \opnorm{v}_{2,1}.
\end{align*}
These imply the desired estimates in the case $m=1$. 
Similarly, we see that 
\begin{align*}
\opnorm{(u'v')(t)}_2^\dag
&\leq \|u'v'\|_{Y^2} + \|(\dt u)'v'\|_{Y^1} + \|u'(\dt v)'\|_{Y^1} \\
&\quad\;
 + \|(\dt^2u)'v'\|_{Y^0} + 2\|(\dt u)'(\dt v)'\|_{Y^0} + \|u'(\dt^2v)'\|_{Y^0} \\
&\lesssim \|u\|_{X^4}\|v\|_{X^3} + \|\dt u\|_{X^3}\|v\|_{X^3} + \|u\|_{X^4}\|\dt v\|_{X^2} \\
&\quad\;
 + \|\dt^2u\|_{X^2}\|v\|_{X^2} + \|\dt u\|_{X^2}\|\dt v\|_{X^2} + \|u\|_{X^4}\|\dt^2v\|_{X^1} \\
&\lesssim \opnorm{u}_{4,2} \opnorm{v}_{3,2},
\end{align*}
which implies the desired estimate in the case $m=2$. 
We then consider the case $m\geq3$. 
\[
\opnorm{(u'v')(t)}_m^\dag \lesssim \sum_{0\leq j\leq m}\sum_{j_1+j_2=j}\|(\dt^{j_1}u)'(\dt^{j_2}v)'\|_{Y^{m-j}}.
\]
We evaluate $I(j_1,j_2;j)=\|(\dt^{j_1}u)'(\dt^{j_2}v)'\|_{Y^{m-j}}$, where $0\leq j\leq m$ and $j_1+j_2=j$. 
In the following calculations, we use Lemma \ref{lem:CalIneq1}. 
\begin{enumerate}
\item[(i)]
The case $j\leq m-3$. 
\begin{align*}
I(j_1,j_2;j) 
&\lesssim \|\dt^{j_1}u\|_{X^{m+1-j}} \|\dt^{j_2}v\|_{X^{m+1-j}} \\
&\lesssim \opnorm{u}_{m+1,m} \opnorm{v}_{m+1,m}.
\end{align*}
\item[(ii)]
The case $j=m-2$. 
\begin{align*}
I(j_1,j_2;j)
&\lesssim 
\begin{cases}
 \|\dt^{m-2}u\|_{X^3} \|v\|_{X^4} &\mbox{for}\quad j_1=j, \\
 \|\dt^{j_1}u\|_{X^4} \|\dt^{j_2}v\|_{X^3} &\mbox{for}\quad j_1\leq j-1
\end{cases} \\
&\lesssim \opnorm{u}_{m+1,m} \opnorm{v}_{m+1,m}.
\end{align*}
\item[(iii)]
The case $j=m-1,m$. 
\begin{align*}
I(j_1,j_2;j)
&\lesssim 
\begin{cases}
 \|\dt^j u\|_{X^{m+1-j}} \|v\|_{X^4} &\mbox{for}\quad j_1=j, \\
 \|u\|_{X^4} \|\dt^j v\|_{X^{m+1-j}} &\mbox{for}\quad j_2=j, \\
 \|\dt^{j_1}u\|_{X^{m+2-j}} \|\dt^{j_2}v\|_{X^{m+2-j}} &\mbox{for}\quad j_1,j_2\leq j-1
\end{cases} \\
&\lesssim \opnorm{u}_{m+1,m} \opnorm{v}_{m+1,m}.
\end{align*}
\end{enumerate}
These imply the desired estimate in the case $m\geq3$. 
\end{proof}

\begin{lemma}[{\cite[Lemma 4.11]{IguchiTakayama2023}}]\label{lem:CalIneq2}
If $\tau|_{s=0}=0$, then we have 
\[
\|\tau u''v''\|_{Y^m} \lesssim
\begin{cases}
 \|\tau'\|_{L^2} \|u\|_{X^4} \|v\|_{X^3} &\mbox{for}\quad m=0, \\
 \|\tau'\|_{L^\infty} \min\{ \|u\|_{X^4} \|v\|_{X^2}, \|u\|_{X^3} \|v\|_{X^3} \} &\mbox{for}\quad m=0, \\
 \min\{ \|\tau'\|_{L^2} \|u\|_{X^4},\|\tau'\|_{L^\infty} \|u\|_{X^3}\} \|v\|_{X^4} &\mbox{for}\quad m=1, \\
 \|\tau'\|_{L^\infty \cap X^{m-1}} \|u\|_{X^{m+2}} \|v\|_{X^{m+2}} &\mbox{for}\quad m\geq2.
\end{cases}
\]
\end{lemma}

\begin{lemma}\label{lem:CalIneqY2}
If $\tau|_{s=0}=0$, then we have 
\[
\opnorm{\tau u''v''}_m ^\dag \lesssim
\begin{cases}
 (\|\tau'\|_{L^\infty} + \|\dt\tau'\|_{L^2}) \opnorm{u}_{4,1} \opnorm{v}_{3,1} &\mbox{for}\quad m=1, \\
 \|(\tau',\dt\tau')\|_{L^2} \opnorm{u}_{4,1} \opnorm{v}_{4,1} &\mbox{for}\quad m=1, \\
 (\|\dt^{m-2}\tau'\|_{L^\infty} + \opnorm{\tau'}_{m-1} + \|\dt^m\tau'\|_{L^2}) \opnorm{u}_{m+2,m} \opnorm{v}_{m+2,m} &\mbox{for}\quad m\geq2.
\end{cases}
\]
\end{lemma}

\begin{proof}
By Lemma \ref{lem:CalIneq2}, we see that 
\begin{align*}
\opnorm{\tau u''v''}_1^\dag
&\leq \|\tau u''v''\|_{Y^1} + \|\tau (\dt u)''v''\|_{Y^0} + \|\tau u''(\dt v)''\|_{Y^0} + \|(\dt\tau)u''v''\|_{Y^0} \\
&\lesssim \|\tau'\|_{L^\infty}( \|u\|_{X^4}\|v\|_{X^3} + \|\dt u\|_{X^3} \|v\|_{X^4} + \|u\|_{X^4}\|\dt v\|_{X^3}) \\
&\quad\;
 + \|\dt\tau'\|_{L^2}\|u\|_{X^4}\|v\|_{X^3} \\
&\lesssim (\|\tau'\|_{L^\infty} + \|\dt\tau'\|_{L^2}) \opnorm{u}_{4,1} \opnorm{v}_{3,1}
\end{align*}
and that 
\begin{align*}
\opnorm{\tau u''v''}_1^\dag
&\lesssim \|\tau'\|_{L^2}( \|u\|_{X^4}\|v\|_{X^4} + \|\dt u\|_{X^3} \|v\|_{X^3} + \|u\|_{X^4}\|\dt v\|_{X^2}) \\
&\quad\;
 + \|\dt\tau'\|_{L^2}\|u\|_{X^4}\|v\|_{X^3} \\
&\lesssim \|(\tau',\dt\tau')\|_{L^2} \opnorm{u}_{4,1} \opnorm{v}_{4,1},
\end{align*}
which imply the desired estimate in the case $m=1$. 
We then consider the case $m\geq2$. 
\[
\opnorm{\tau u''v''}_m^\dag \lesssim \sum_{0\leq j\leq m}\sum_{j_0+j_1+j_2=j}\|(\dt^{j_0}\tau)(\dt^{j_1}u)''(\dt^{j_2}v)''\|_{Y^{m-j}}.
\]
We evaluate $I(j_0,j_1,j_2;j)=\|(\dt^{j_0}\tau)(\dt^{j_1}u)''(\dt^{j_2}v)''\|_{Y^{m-j}}$, where $0\leq j\leq m$ and $j_0+j_1+j_2=j$. 
In the following calculations, we use Lemma \ref{lem:CalIneq2}. 
\begin{enumerate}
\item[(i)]
The case $j\leq m-2$. 
\begin{align*}
I(j_0,j_1,j_2;j)
&\lesssim \|\dt^{j_0}\tau'\|_{L^\infty\cap X^{m-1-j}} \|\dt^{j_1}u\|_{X^{m+2-j}} \|\dt^{j_2}v\|_{X^{m+2-j}} \\
&\lesssim (\|\dt^{m-2}\tau'\|_{L^\infty}+\opnorm{\tau'}_{m-1})\opnorm{u}_{m+2,m} \opnorm{v}_{m+2,m},
\end{align*}
where we used $\|f\|_{L^\infty} \lesssim \|f\|_{X^2}$. 
\item[(ii)]
The case $j=m-1$. 
\begin{align*}
I(j_0,j_1,j_2;j)
&\lesssim 
\begin{cases}
 \|\tau'\|_{L^\infty} \|\dt^{m-1}u\|_{X^3} \|v\|_{X^4} &\mbox{for}\quad j_1=j, \\
 \|\tau'\|_{L^\infty} \|u\|_{X^4} \|\dt^{m-1}v\|_{X^3} &\mbox{for}\quad j_2=j, \\
 \|\dt^{j_0}\tau'\|_{L^2} \|\dt^{j_1}u\|_{X^4} \|\dt^{j_2}v\|_{X^4} &\mbox{for}\quad j_1,j_2\leq j-1
\end{cases} \\
&\lesssim (\|\tau'\|_{L^\infty}+\opnorm{\tau'}_{m-1})\opnorm{u}_{m+2,m} \opnorm{v}_{m+2,m}.
\end{align*}
\item[(iii)]
The case $j=m$. 
\begin{align*}
I(j_0,j_1,j_2;j)
&\lesssim 
\begin{cases}
 \|\tau'\|_{L^\infty} \|\dt^mu\|_{X^2}\|v\|_{X^4} &\mbox{for}\quad j_1=j, \\
 \|\tau'\|_{L^\infty} \|u\|_{X^4}\|\dt^mv\|_{X^2} &\mbox{for}\quad j_2=j, \\
 \|\tau'\|_{L^\infty} \|\dt^{j_1}u\|_{X^3}\|\dt^{j_2}v\|_{X^3} &\mbox{for}\quad j_0=0, \ j_1,j_2\leq j-1, \\
 \|\dt^{j_0}\tau'\|_{L^2} \|\dt^{j_1}u\|_{X^4}\|\dt^{j_2}v\|_{X^3 } &\mbox{for}\quad j_0\geq1, \ j_1\leq j_2\leq j-1, \\
 \|\dt^{j_0}\tau'\|_{L^2} \|\dt^{j_1}u\|_{X^3}\|\dt^{j_2}v\|_{X^4 } &\mbox{for}\quad j_0\geq1, \ j_2\leq j_1\leq j-1
\end{cases} \\
&\lesssim (\|\tau'\|_{L^\infty} + \opnorm{\tau'}_{m-1} + \|\dt^m\tau'\|_{L^2}) \opnorm{u}_{m+2,m} \opnorm{v}_{m+2,m}.
\end{align*}
\end{enumerate}
By noting $\|\tau'\|_{L^\infty} \lesssim \|\dt^{m-2}\tau'\|_{L^\infty} + \opnorm{\tau'}_{m-1}$, 
these imply the desired estimate in the case $m\geq2$. 
\end{proof}

\begin{lemma}\label{lem:CalIneqY3}
If $\tau|_{s=0}=0$, then we have 
\[
\|s^\frac12\tau u''v''\|_{L^1} 
\lesssim \min\{ \|\tau'\|_{L^\infty}\|u\|_{X^2}\|v\|_{X^3}, \|\tau'\|_{L^2}\|u\|_{X^2}\|v\|_{X^4}, \|\tau'\|_{L^2}\|u\|_{X^3}\|v\|_{X^3} \}.
\]
\end{lemma}

\begin{proof}
It is sufficient to note that $|\tau(s)| \leq s^{1-\frac1p}\|\tau'\|_{L^p}$ for $1\leq p\leq \infty$. 
\end{proof}

\section{Estimates for initial values and compatibility conditions I}\label{sect:EstIV1}
We consider the initial boundary value problem \eqref{LS}. 
Let $\bm{u}$ be a smooth solution to the problem \eqref{LS} and put $\bm{u}_j^\mathrm{in}=(\dt^j\bm{u})|_{t=0}$ for $j=0,1,2,\ldots$. 
By applying $\dt^j$ to the hyperbolic system in \eqref{LS} and putting $t=0$, we see that the $\{\bm{u}_j^\mathrm{in}\}$ are calculated inductively by 
\begin{equation}\label{HOIV}
\bm{u}_{j+2}^\mathrm{in}
= \sum_{k=0}^j\binom{j}{k}\bigl\{ ((\dt^{j-k}A)|_{t=0}(\bm{u}_k^\mathrm{in})')' + (\dt^{j-k}Q)|_{t=0}(\bm{u}_k^\mathrm{in})'(1) \bigr\} + (\dt^j\bm{f})|_{t=0}
\end{equation}
for $j=0,1,2,\ldots$. 
Then, by applying $\dt^j$ to the boundary condition in \eqref{LS} and putting $t=0$, we obtain 
\begin{equation}\label{CC1}
\bm{u}_j^\mathrm{in}(1)=0
\end{equation}
for $j=0,1,2,\ldots$. 
These are necessary conditions that the data $(\bm{u}_0^\mathrm{in},\bm{u}_1^\mathrm{in},\bm{f})$ should satisfy 
for the existence of a regular solution to the problem \eqref{LS} and are known as compatibility conditions. 
To state the conditions more precisely, we need to evaluate the initial values $\{\bm{u}_j^\mathrm{in}\}$. 
Although it is sufficient to evaluate $\dt^j\bm{u}$ only at time $t=0$, we will evaluate them at general time $t$.

\begin{lemma}\label{lem:EstIV}
Let $m\geq2$ be an integer and assume that Assumptions \ref{ass:BEE} and \ref{ass:HOEE} are satisfied with a positive constant $M_0$ 
and that $\bm{f}\in\mathscr{X}_T^{m-2}$. 
Then, there exists a positive constant $C_0$ depending only on $m$ and $M_0$ such that if $\bm{u}$ is a solution to \eqref{LS}, 
then we have $\opnorm{\bm{u}(t)}_m \leq C_0( \opnorm{\bm{u}(t)}_{m,1} + \opnorm{\bm{f}(t)}_{m-2} )$. 
\end{lemma}

\begin{proof}
Let $0\leq j\leq m-2$. 
It follows from \eqref{LS} that 
\begin{align*}
\opnorm{\bm{u}}_{m,j+2}
&\leq \opnorm{\bm{u}}_{m,1} + \opnorm{\ddot{\bm{u}}}_{m-2,j} \\
&\leq \opnorm{\bm{u}}_{m,1} + \opnorm{(A\bm{u}')'}_{m-2,j} + \opnorm{Q\bm{u}'(1,t)}_{m-2,j} + \opnorm{\bm{f}}_{m-2,j}.
\end{align*}
We note that by Assumption \ref{ass:BEE} we have $A|_{s=0}=O$. 
Therefore, by Lemmas \ref{lem:EstAu} and \ref{lem:EstAu2} we see that 
\begin{align*}
\opnorm{(A\bm{u}')'}_{m-2,j}
&\lesssim
\begin{cases}
 \|A'\|_{L^\infty} \opnorm{\bm{u}}_{m,j} &\mbox{for}\quad m=2, \\
 \|A'\|_{X^2} \opnorm{\bm{u}}_{m,j} &\mbox{for}\quad m=3, j=0, \\
 \opnorm{A'}_{2,1} \opnorm{\bm{u}}_{m,j} &\mbox{for}\quad m=3, j=1, \\
 \opnorm{A'}_{m-2,j} \opnorm{\bm{u}}_{m,j} &\mbox{for}\quad m\geq4
\end{cases} \\
&\lesssim \opnorm{\bm{u}}_{m,j}.
\end{align*}
By the standard Sobolev embedding theorem, we see also that 
\begin{align*}
\opnorm{Q\bm{u}'(1,t)}_{m-2,j}
&\lesssim \opnorm{Q}_{m-2} \sum_{k=0}^j\|\dt^k\bm{u}\|_{X^2} 
 \lesssim \opnorm{\bm{u}}_{m,j}.
\end{align*}
Therefore, we get $\opnorm{\bm{u}}_{m,j+2} \lesssim \opnorm{\bm{u}}_{m,j} + \opnorm{\bm{f}}_{m-2}$ for $0\leq j\leq m-2$. 
Using this inductively on $j$, we obtain the desired estimate. 
\end{proof}

Under the same assumptions in Lemma \ref{lem:EstIV}, we see that if the initial data satisfy $\bm{u}_j^\mathrm{in} \in X^{m-j}$ for $j=0,1$, 
then the initial values $\{\bm{u}_j^\mathrm{in}\}$ satisfy $\bm{u}_j^\mathrm{in} \in X^{m-j}$ for $j=0,1,\ldots,m$, so that 
their boundary values $\bm{u}_j^\mathrm{in}(1)$ are defined for $j=0,1,\ldots,m-1$.

\begin{definition}\label{def:CC1}
Let $m\geq1$ be an integer. 
We say that the data $(\bm{u}_0^\mathrm{in},\bm{u}_1^\mathrm{in},\bm{f})$ 
for the initial boundary value problem \eqref{LS} satisfy the compatibility conditions up to order $m-1$ if \eqref{CC1} holds for any $j=0,1,\ldots,m-1$. 
\end{definition}


\section{Energy estimates}\label{sect:EE}
Difficulties showing a well-posedness of the initial boundary value problem \eqref{LS} is caused not only by the degeneracy of the matrix $A(s,t)$ 
at the end $s=0$ but also by the localized term $Q(s,t)\bm{u}'(1,t)$. 
The first difficulty could be overcame by using weights in the norm of Sobolev spaces, see Takayama \cite{Takayama2018}, 
whereas the second one will be treated by regularizing the hyperbolic system. 
In this paper, we adopt the regularized problem \eqref{rLS}. 
As we will see later, in the case $\varepsilon>0$ the regularized term $\varepsilon s\dot{\bm{u}}'$ makes the localized term 
$Q(s,t)\bm{u}'(1,t)$ to be of lower order. 
Before giving energy estimates for the solution, we recall the following lemma.

\begin{lemma}[{\cite[Lemma 6.1]{IguchiTakayama2023}}]\label{lem:equiv}
Let $M_0$ be a positive constant. 
There exists a constant $C_0=C(M_0)>1$ such that if a symmetric matrix $A(s)$ satisfies 
$M_0^{-1}s\mathrm{Id}\leq A(s)\leq M_0s\mathrm{Id}$ and $|A'(s)|\leq M_0$ for $s\in[0,1]$, then we have the equivalence 
\[
C_0^{-1}( \|s\bm{u}''\|_{L^2}^2+\|\bm{u}'\|_{L^2}^2 ) \leq \|(A\bm{u}')'\|_{L^2}^2 \leq C_0( \|s\bm{u}''\|_{L^2}^2+\|\bm{u}'\|_{L^2}^2 ).
\]
\end{lemma}

The following proposition gives a basic energy estimate for the solution of the problem \eqref{rLS}.

\begin{proposition}\label{prop:BEE}
Let $T$, $M_0$, and $M_1$ be positive constants and suppose that Assumption \ref{ass:BEE} is satisfied and that 
$\bm{f}\in C^0([0,T];L^2)$ and $\dt\bm{f}\in L^1(0,T;L^2)$. 
Then, there exist positive constants $C_0=C_0(M_0)$ and $\gamma_1=\gamma_1(M_0,M_1)$ such that the solution $\bm{u}\in\mathscr{X}_T^{2,*}$ to the problem 
\eqref{rLS} satisfies an additional regularity $\varepsilon\bm{u}'|_{s=1}\in H^1(0,T)$ and an energy estimate 
\begin{align}\label{BEE}
& I_{\gamma,t}( \opnorm{\bm{u}(\cdot)}_2 ) + \sqrt{\varepsilon}|\bm{u}'(1,\cdot)|_{H_\gamma^1(0,t)} \\
&\leq C_0 \left\{ \|\bm{u}(0)\|_{X^2}+\|\dot{\bm{u}}(0)\|_{X^1}+\|\bm{f}(0)\|_{L^2} + S_{\gamma,t}^*( \|\dt\bm{f}(\cdot)\|_{L^2} ) \right\} \nonumber
\end{align}
for any $t\in[0,T]$, $\gamma\geq\gamma_1$, and $\varepsilon\in[0,1]$. 
\end{proposition}

\begin{proof}
In the following calculations, we simply denote by $C_0$ the constant depending only on $M_0$ and by $C_1$ the constant depending also on $M_1$. 
These constants may change from line to line.

We first suppose that the solution $\bm{u}$ satisfies $\bm{u}\in C^2([0,T]; X^2)$. 
Then, we see that 
\begin{align*}
& \frac{\mathrm{d}}{\mathrm{d}t}\left\{ (A\dot{\bm{u}}',\dot{\bm{u}}')_{L^2} + \|(A\bm{u}')'\|_{L^2}^2 \right\}
 - \left\{ ((\dt A)\dot{\bm{u}}',\dot{\bm{u}}')_{L^2} + 2(((\dt A)\bm{u}')',(A\bm{u}')')_{L^2} \right\} \\
&= 2(A\dot{\bm{u}}',\ddot{\bm{u}}')_{L^2} + 2((A\dot{\bm{u}}')',(A\bm{u}')')_{L^2} \\
&= -2((A\dot{\bm{u}}')',\ddot{\bm{u}}-(A\bm{u}')')_{L^2} \\
&= -2((A\dot{\bm{u}}')',\varepsilon s\dot{\bm{u}}'+Q\bm{u}'(1,t)+\bm{f})_{L^2},
\end{align*}
where we used the boundary condition $\ddot{\bm{u}}|_{s=1}=0$. 
Here, by integration by parts we have 
%
%
%
\[
2((A\dot{\bm{u}}')',s\dot{\bm{u}}')_{L^2}
= (\dot{\bm{u}}'\cdot A\dot{\bm{u}}')|_{s=1} - (\dot{\bm{u}}',(A-sA')\dot{\bm{u}}')_{L^2}.
\]
We see also that 
\begin{align*}
((A\dot{\bm{u}}')',Q\bm{u}'(1,t)+\bm{f})_{L^2}
&= ((A\dot{\bm{u}}')',sQ\bm{u}'+\bm{f})_{L^2} + ((A\dot{\bm{u}}')',Q(\bm{u}'(1,t)-s\bm{u}'))_{L^2} \\
&= \frac{\mathrm{d}}{\mathrm{d}t}((A\bm{u}')',sQ\bm{u}'+\bm{f})_{L^2} - (((\dt A)\bm{u}')',sQ\bm{u}'+\bm{f})_{L^2} \\
&\quad\;
 - ((A\bm{u}')',\dt(sQ\bm{u}'+\bm{f}))_{L^2} - (A\dot{\bm{u}}',(Q(\bm{u}'(1,t)-s\bm{u}'))')_{L^2}.
\end{align*}
In view of these identities, we introduce an energy functional $\mathscr{E}_2(t)$ by 
\[
\mathscr{E}_2(t) = (A\dot{\bm{u}}',\dot{\bm{u}}')_{L^2} + \|(A\bm{u}')'\|_{L^2}^2 + 2((A\bm{u}')',sQ\bm{u}'+\bm{f})_{L^2}
 + \lambda( \|\dot{\bm{u}}\|_{L^2}^2 + \|\bm{u}\|_{X^1}^2 ),
\]
where $\lambda>0$ is a parameter. 
By Lemma \ref{lem:equiv}, it is easy to check that there exists a sufficiently large $\lambda_0=\lambda(M_0)$ such that 
if we choose $\lambda=\lambda_0$, then we have 
\begin{equation}\label{equiv}
\begin{cases}
 \mathscr{E}_2(t) \leq C_0( \opnorm{ \bm{u}(t) }_{2,*}^2 + \|\bm{f}(t)\|_{L^2}^2 ), \\
 \opnorm{ \bm{u}(t) }_{2,*}^2 \leq C_0 ( \mathscr{E}_2(t) + \|\bm{f}(t)\|_{L^2}^2 )
\end{cases}
\end{equation}
for $0\leq t\leq T$. 
Moreover, we have 
\begin{align*}
\frac{\mathrm{d}}{\mathrm{d}t}\mathscr{E}_2(t) + \varepsilon(\dot{\bm{u}}'\cdot A\dot{\bm{u}}')|_{s=1}
&= ((\dt A)\dot{\bm{u}}',\dot{\bm{u}}')_{L^2} + 2(((\dt A)\bm{u}')',(A\bm{u}')'+sQ\bm{u}'+\bm{f})_{L^2} \\
&\quad\;
 + 2((A\bm{u}')',\dt(sQ\bm{u}'+\bm{f}))_{L^2} + 2(A\dot{\bm{u}}',(Q(\bm{u}'(1,t)-s\bm{u}'))')_{L^2} \\
&\quad\;
 + \varepsilon(\dot{\bm{u}}',(A-sA')\dot{\bm{u}}')_{L^2} + \lambda\frac{\mathrm{d}}{\mathrm{d}t}( \|\dot{\bm{u}}\|_{L^2}^2 + \|\bm{u}\|_{X^1}^2 ) \\
&\leq C_1 ( \opnorm{ \bm{u}(t) }_{2,*}^2 + \opnorm{ \bm{u}(t) }_{2,*} \|\bm{f}(t)\|_{L^2} ) + C_0\opnorm{ \bm{u}(t) }_{2,*}\|\dot{\bm{f}}(t)\|_{L^2},
\end{align*}
where we used Lemma \ref{lem:equiv} together with $|\dt A(s,t)| \leq M_1s$, which comes directly from Assumption \ref{ass:BEE}. 
Therefore, for any $\gamma>0$ we have 
\begin{align*}
& \frac{\mathrm{d}}{\mathrm{d}t}\{ \mathrm{e}^{-2\gamma t}\mathscr{E}_2(t) \}
 + 2\gamma\mathrm{e}^{-2\gamma t}\mathscr{E}_2(t) + \varepsilon M_0^{-1}\mathrm{e}^{-2\gamma t}|\dot{\bm{u}}'(1,t)|^2 \\
&\leq \mathrm{e}^{-2\gamma t}\{ C_1 ( \opnorm{ \bm{u}(t) }_{2,*}^2 + \opnorm{ \bm{u}(t) }_{2,*} \|\bm{f}(t)\|_{L^2}) 
 + C_0\opnorm{ \bm{u}(t) }_{2,*} \|\dot{\bm{f}}(t)\|_{L^2} \}.
\end{align*}
Integrating this with respect to $t$ and using \eqref{equiv}, $|f|_{L_\gamma^2(0,t)} \leq \gamma^{-\frac12}I_{\gamma,t}(f)$, and 
\[
\left| \int_0^t\mathrm{e}^{-2\gamma t'}f(t')\varphi(t')\mathrm{d}t' \right| \leq I_{\gamma,t}(f)S_{\gamma,t}^*(\varphi),
\]
we obtain 
\begin{align*}
& \mathrm{e}^{-2\gamma t}\opnorm{ \bm{u}(t) }_{2,*}^2 + 2\gamma \int_0^t \mathrm{e}^{-2\gamma t'}\opnorm{ \bm{u}(t') }_{2,*}^2\mathrm{d}t'
 + \varepsilon|\dot{\bm{u}}'(1,\cdot)|_{L_\gamma^2(0,t)}^2 \\
&\leq C_0\{ \opnorm{ \bm{u}(0) }_{2,*}^2 + I_{\gamma,t}(\|\bm{f}(\cdot)\|_{L^2})^2
 + I_{\gamma,t}( \opnorm{ \bm{u}(\cdot) }_{2,*}) S_{\gamma,t}^*(\|\dot{\bm{f}}(\cdot)\|_{L^2}) \} \\
&\quad\;
 + \gamma S_{\gamma,t}^*(\|\bm{f}(\cdot)\|_{L^2})^2 + C_1\gamma^{-1}I_{\gamma,t}( \opnorm{ \bm{u}(\cdot) }_{2,*} )^2.
\end{align*}
As was shown by Iguchi and Lannes \cite[Lemma 2.16]{IguchiLannes2021}, we have also 
\begin{equation}\label{S*}
\begin{cases}
 \gamma S_{\gamma,t}^*(\|\bm{f}(\cdot)\|_{L^2}) \leq I_{\gamma,t}(\|\bm{f}(\cdot)\|_{L^2})
  \leq C( \|\bm{f}(0)\|_{L^2} + S_{\gamma,t}^*(\|\dot{\bm{f}}(\cdot)\|_{L^2}) ), \\
 |\bm{u}'(1,t)|_{L_\gamma^2(0,t)} \leq C( \gamma^{-\frac12}|\bm{u}'(1,0)|+\gamma^{-1}|\dot{\bm{u}}'(1,t)|_{L_\gamma^2(0,t)})
\end{cases}
\end{equation}
with an absolute constant $C>0$. 
Therefore, by choosing $\gamma_1$ so large that $C_1\gamma_1^{-1}\leq\frac12$, we obtain 
\[
I_{\gamma,t}( \opnorm{ \bm{u}(\cdot) }_{2,*} ) + \sqrt{\varepsilon}|\bm{u}'(1,\cdot)|_{H_\gamma^1(0,t)}
\leq C_0\{ \opnorm{ \bm{u}(0) }_{2,*} + \|\bm{f}(0)\|_{L^2} + S_{\gamma,t}^*(\|\dot{\bm{f}}(\cdot)\|_{L^2}) \}.
\]
By using the hyperbolic system for $\bm{u}$, we have $\|\ddot{\bm{u}}\|_{L^2} \leq C_0( \opnorm{\bm{u}}_{2,*}+\|\bm{f}\|_{L^2})$. 
These estimates imply the desired one.

In the case $\bm{u}\in\mathscr{X}_T^{2,*}$, we use a mollifier $\rho_\epsilon*$ with respect to $t$ with a kernel 
$\rho_\epsilon(t)=\frac{1}{\epsilon}\rho(\frac{t}{\epsilon})$ satisfying $\rho\in C_0^\infty(\mathbb{R})$, $\operatorname{supp}\rho \subset (-1,0)$, and 
$\int_\mathbb{R}\rho(t)\mathrm{d}t=1$. 
The procedure is standard so we omit the details. 
\end{proof}

We then prepare estimates for the solution $\bm{u}$ to the problem \eqref{LS}, which convert spatial derivatives into time derivatives 
by using the hyperbolic system in \eqref{LS}.

\begin{lemma}\label{lem:Estu}
Let $T$ and $M_0$ be positive constants and $m\geq2$ an integer. 
Suppose that Assumptions \ref{ass:BEE} and \ref{ass:HOEE} are satisfied and that $\bm{f}\in\mathscr{X}_T^{m-2,*}$ in the case $m\geq3$. 
Then, there exists a positive constant $C_0=C_0(m,M_0)$ such that the solution $\bm{u}\in\mathscr{X}_T^{m,*}$ to the problem \eqref{rLS} satisfies 
\[
\opnorm{\bm{u}(t)}_m \leq C_0 \bigl( \opnorm{\dt^{m-2}\bm{u}(t)}_2 + \opnorm{\bm{u}(t)}_{m-1} + \opnorm{\bm{f}(t)}_{m-2,*} \bigr)
\]
for any $t\in[0,T]$, where we use a notational convention $\opnorm{\cdot}_{0,*}=0$. 
\end{lemma}

\begin{proof}
It is sufficient to evaluate $\|\dt^{m-j}\bm{u}\|_{X^j}$ for $3\leq j\leq m$. 
To this end, we use the identity 
\begin{align}\label{NormDeco}
\|\dt^{m-j}\bm{u}\|_{X^j}^2
&= \|\dt^{m-j}\bm{u}\|_{L^2}^2 + \|\dt^{m-j}\bm{u}'\|_{X^{j-2}}^2 + \|s^\frac{j}{2}\ds^j\dt^{m-j}\bm{u}\|_{L^2}^2.
\end{align}
Obviously, we have $\|\dt^{m-j}\bm{u}\|_{L^2} \leq \opnorm{\bm{u}}_{m-1}$. 
To evaluate the second term in the right-hand side, we introduce a matrix valued function $A_0(s,t)=\frac{1}{s}A(s,t)=\mathscr{M}(A'(\cdot,t))(s)$, 
which is symmetric and satisfies 
\[
\begin{cases}
 M_0^{-1}\mathrm{Id} \leq A_0(s,t) \leq M_0\mathrm{Id}, \\
 \opnorm{A_0(t)}_{m-2}+\opnorm{A_0(t)}_{2,*} \leq 2M_0
\end{cases}
\]
for any $(s,t)\in(0,1)\times(0,T)$, where we used Lemma \ref{lem:WEM1} to derive the above estimates. 
Therefore, by Lemma \ref{lem:EstCompFunc1} we obtain $\opnorm{A_0^{-1}(t)}_{m-2}+\opnorm{A_0^{-1}(t)}_{2,*} \leq C_0$ 
with a constant $C_0$ depending only on $m$ and $M_0$. 
Moreover, by the hyperbolic system in \eqref{rLS} we have $\bm{u}'=A_0^{-1}\mathscr{M}\bm{F}$ with 
$\bm{F}= \ddot{\bm{u}}-(Q\bm{u}'(1,t)+\varepsilon s\dot{\bm{u}}'+\bm{f})$, so that by Lemmas \ref{lem:embedding2}, \ref{lem:Algebra}, and \ref{lem:WEM1}
\begin{align}\label{EstMidOrd}
\|\dt^{m-j}\bm{u}'\|_{X^{j-2}}
&\lesssim 
 \begin{cases}
  \|\bm{F}\|_{X^1} &\mbox{for}\quad m=j=3, \\
  \|\dt\bm{F}\|_{X^1} + \|\mathscr{M}\bm{F}\|_{X^2} &\mbox{for}\quad m=4, j=3, \\
  \|\dt^{m-j}\bm{F}\|_{X^{j-2}}+\opnorm{\bm{F}}_{m-3} &\mbox{for}\quad m\geq5, j=3 \mbox{ or } m\geq j\geq 4
 \end{cases} \\
&\lesssim \|\dt^{m-(j-2)}\bm{u}\|_{X^{j-2}} + \|\dt^{m-(j-1)}\bm{u}\|_{X^{j-1}} + \opnorm{\bm{u}}_{m-1} + \opnorm{\bm{f}}_{m-2,*}, \nonumber
\end{align}
where we used the identity $\mathscr{M}(s\dot{\bm{u}}')=\dot{\bm{u}}-\mathscr{M}\dot{\bm{u}}$ to evaluate $\|\mathscr{M}\bm{F}\|_{X^2}$ in the case $m=4,j=3$. 

We proceed to evaluate the highest order term in \eqref{NormDeco}. 
Applying $\ds^{j-2}\dt^{m-j}$ to the hyperbolic system in \eqref{rLS}, we obtain 
\begin{align*}
s|\ds^j\dt^{m-j}\bm{u}|
& \lesssim |[\ds^{j-1},A]\dt^{m-j}\bm{u}'| + |\ds^{j-2}([\dt^{m-j},A]\bm{u}')'| \\
&\quad\;
 + |\ds^{j-2}\dt^{m-j}(\ddot{\bm{u}}-(Q\bm{u}'(1,t)+\varepsilon s\dot{\bm{u}}'+\bm{f}))|.
\end{align*}
Therefore, by Lemmas \ref{lem:commutator} we obtain 
\begin{align}\label{CommEst}
\|s^\frac{j}{2}\ds^{j}\dt^{m-j}\bm{u}\|_{L^2}
&\lesssim \|s^\frac{j-2}{2}[\ds^{j-1},A]\dt^{m-j}\bm{u}'\|_{L^2} + \|s^\frac{j-2}{2}\ds^{j-2}([\dt^{m-j},A]\bm{u}')'\|_{L^2} \\
&\quad\;
 + \|s^\frac{j-2}{2}\ds^{j-2}\dt^{m-j}(\ddot{\bm{u}}-(Q\bm{u}'(1,t)+\varepsilon s\dot{\bm{u}}'+\bm{f}))\|_{L^2} \nonumber \\
&\lesssim \|A'\|_{X^{j-2\vee2}} \|\dt^{m-j}\bm{u}'\|_{X^{j-2}} + \|([\dt^{m-j},A]\bm{u}')'\|_{X^{j-2}} \nonumber \\
&\quad\;
 + \|\dt^{m-j}(\ddot{\bm{u}}-(Q\bm{u}'(1,t)+\varepsilon s\dot{\bm{u}}'+\bm{f})\|_{X^{j-2}}. \nonumber
\end{align}
As for the second term in the right-hand side, it is sufficient to evaluate it in the case $3\leq j\leq m-1$. 
In the case $m\geq5$ we see that 
\begin{align*}
\|([\dt^{m-j},A]\bm{u}')'\|_{X^{j-2}}
&\lesssim \|((\dt^{m-j}A)\bm{u}')'\|_{X^{j-2}} + \sum_{j_1+j_2=m-j-2} \|((\dt^{j_1+1}A)(\dt^{j_2+1}\bm{u})')'\|_{X^{j-2}} \\
&\lesssim \|\dt^{m-j}A'\|_{X^{j-2}}\|\bm{u}\|_{X^{j\vee4}} + \sum_{j_1+j_2=m-j-2} \|\dt^{j_1+1}A'\|_{X^{j-2\vee2}} \|\dt^{j_2+1}\bm{u}\|_{X^{j}} \\
&\lesssim \opnorm{A'}_{m-2} \opnorm{\bm{u}}_{m-1}.
\end{align*}
In the case $m=4$ we may assume $j=3$ so that by \eqref{AuxEst1} we evaluate it as 
$\|([\dt^{m-j},A]\bm{u}')'\|_{X^{j-2}}=\|((\dt A)\bm{u}')'\|_{X^1} \lesssim \|\dt A'\|_{X^1}\|\bm{u}'\|_{X^2}$. 
To evaluate $\|\bm{u}'\|_{X^2}$ we modify \eqref{EstMidOrd} slightly to get 
$\|\bm{u}'\|_{X^2} \lesssim \|\dt^2\bm{u}\|_{X^2}+\opnorm{\bm{u}}_3+\opnorm{\bm{f}}_{2,*}$. 
The last term in \eqref{CommEst} can be easily evaluated so that we get 
\begin{align*}
\|s^\frac{j}{2}\ds^{j}\dt^{m-j}\bm{u}\|_{L^2}
&\lesssim \|\dt^{m-(j-2)}\bm{u}\|_{X^{j-2}} + \|\dt^{m-(j-1)}\bm{u}\|_{X^{j-1}} \\
&\quad\; 
 + \|\dt^{m-j}\bm{u}'\|_{X^{j-2}} + \opnorm{\bm{u}}_{m-1} + \opnorm{\bm{f}}_{m-2,*}.
\end{align*}

Summarizing the above estimates, we obtain 
\[
\|\dt^{m-j}\bm{u}\|_{X^j}
\lesssim \|\dt^{m-(j-2)}\bm{u}\|_{X^{j-2}} + \|\dt^{m-(j-1)}\bm{u}\|_{X^{j-1}} + \opnorm{\bm{u}}_{m-1} + \opnorm{\bm{f}}_{m-2,*}
\]
for $3\leq j\leq m$. 
Using this inductively, we finally obtain the desired estimate. 
\end{proof}

The following proposition gives higher order energy estimates for the solution of the problem \eqref{rLS}.

\begin{proposition}\label{prop:HOEE}
Let $T$, $M_0$, and $M_1$ be positive constants and $m\geq2$ an integer. 
Suppose that Assumptions \ref{ass:BEE} and \ref{ass:HOEE} are satisfied and that 
$\bm{f}\in\mathscr{X}_T^{m-2}$ and $\dt^{m-1}\bm{f}\in L^1(0,T;L^2)$. 
Then, there exist positive constants $C_0=C_0(m,M_0)$ and $\gamma_1=\gamma_1(m,M_0,M_1)$ such that the solution $\bm{u}\in\mathscr{X}_T^{m,*}$ to the problem 
\eqref{rLS} satisfies an additional regularity $\varepsilon\bm{u}'|_{s=1}\in H^{m-1}(0,T)$ and an energy estimate 
\begin{align}\label{HOEE}
& I_{\gamma,t}( \opnorm{\bm{u}(\cdot)}_m ) + \sqrt{\varepsilon}|\bm{u}'(1,\cdot)|_{H_\gamma^{m-1}(0,t)} \\
&\leq C_0 \left\{ \|\bm{u}(0)\|_{X^m}+\|\dot{\bm{u}}(0)\|_{X^{m-1}} + I_{\gamma,t}( \opnorm{\bm{f}(\cdot)}_{m-2} )
 + S_{\gamma,t}^*( \|\dt^{m-1}\bm{f}(\cdot)\|_{L^2} ) \right\} \nonumber
\end{align}
for any $t\in[0,T]$, $\gamma\geq\gamma_1$, and $\varepsilon\in[0,1]$. 
\end{proposition}

\begin{proof}
In the following calculations, we simply denote by $C_0$ the constant depending only on $M_0$ and by $C_1$ the constant depending also on $M_1$. 
These constants may change from line to line. 
Putting $\bm{v}=\dt^{m-2}\bm{u}$, we see that $\bm{v}$ solves 
\begin{equation}\label{drLS}
\begin{cases}
 \ddot{\bm{v}}=(A\bm{v}')'+Q\bm{v}'(1,t)+\varepsilon s\dot{\bm{v}}'+\bm{f}_m &\mbox{in}\quad (0,1)\times(0,T), \\
 \bm{v}=\bm{0} &\mbox{on}\quad \{s=1\}\times(0,T),
\end{cases}
\end{equation}
where $\bm{f}_m = \dt^{m-2}\bm{f} + ([\dt^{m-2},A]\bm{u}')'+[\dt^{m-2},Q]\bm{u}'(1,t)$. 
Applying Proposition \ref{prop:BEE} we obtain 
\begin{align*}
& I_{\gamma,t}( \opnorm{\dt^{m-2}\bm{u}(\cdot)}_2 ) + \sqrt{\varepsilon}|\dt^{m-2}\bm{u}'(1,\cdot)|_{H_\gamma^1(0,t)} \\
&\leq C_0 \left\{ \opnorm{\dt^{m-2}\bm{u}(0)}_{2,*} + \|\bm{f}_m(0)\|_{L^2} + S_{\gamma,t}^*( \|\dt\bm{f}_m(\cdot)\|_{L^2} ) \right\}.
\end{align*}
Here, by the first equation in \eqref{drLS} together with Lemmas \ref{lem:embedding2} and \ref{lem:EstAu} 
we get $\|\bm{f}_m(0)\|_{L^2} \leq C_0\opnorm{\bm{u}(0)}_m$. 
By Lemma \ref{lem:Comm} we get also $\|\dt\bm{f}_m\|_{L^2}\leq \|\dt^{m-1}\bm{f}\|_{L^2} + C_1\opnorm{\bm{u}}_m$. 
These estimates and Lemma \ref{lem:Estu} imply 
\begin{align*}
& I_{\gamma,t}( \opnorm{\bm{u}(\cdot)}_m ) + \sqrt{\varepsilon}|\dt^{m-2}\bm{u}'(1,\cdot)|_{H_\gamma^1(0,t)} \\
&\leq C_0 \left\{ \opnorm{\bm{u}(0)}_m +  I_{\gamma,t}( \opnorm{\bm{f}(\cdot)}_{m-2} ) + I_{\gamma,t}( \opnorm{\bm{u}(\cdot)}_{m-1} )
 + S_{\gamma,t}^*( \|\dt^{m-1}\bm{f}(\cdot)\|_{L^2} ) \right\} \\
&\quad\;
 + C_1S_{\gamma,t}^*( \opnorm{\bm{u}(\cdot)}_m ).
\end{align*}
As was shown by Iguchi and Lannes \cite[Lemma 2.16]{IguchiLannes2021}, we have also 
\[
\begin{cases}
 I_{\gamma,t}( \opnorm{\bm{u}(\cdot)}_{m-1} )
  \leq C( \opnorm{ \bm{u}(0) }_m + S_{\gamma,t}^*( \opnorm{\bm{u}(\cdot)}_m ), \\
 |\bm{u}'(1,t)|_{H_\gamma^{m-2}(0,t)} \leq C( \gamma^{-\frac12}\opnorm{ \bm{u}(0) }_m + \gamma^{-1}|\bm{u}'(1,t)|_{H_\gamma^{m-1}(0,t)})
\end{cases}
\]
with an absolute constant $C>0$. 
Since $S_{\gamma,t}^*( \opnorm{\bm{u}(\cdot)}_m ) \leq \gamma^{-1}I_{\gamma,t}( \opnorm{\bm{u}(\cdot)}_m )$, 
by choosing $\gamma_1$ so large that $C_1\gamma_1^{-1}\ll1$ we obtain 
\begin{align*}
& I_{\gamma,t}( \opnorm{\bm{u}(\cdot)}_m ) + \sqrt{\varepsilon}|\bm{u}'(1,\cdot)|_{H_\gamma^{m-1}(0,t)} \\
&\leq C_0 \left\{ \opnorm{\bm{u}(0)}_m + I_{\gamma,t}( \opnorm{\bm{f}(\cdot)}_{m-2} )
 + S_{\gamma,t}^*( \|\dt^{m-1}\bm{f}(\cdot)\|_{L^2} ) \right\}
\end{align*}
for any $t\in[0,T]$, $\gamma\geq\gamma_1$, and $\varepsilon\in[0,1]$. 
This estimate and Lemma \ref{lem:EstIV} give the desired one. 
\end{proof}

\section{Existence of solutions I}\label{sect:Proof1}
In this section we prove Theorem \ref{Th1}. 
To this end, we first consider the initial boundary value problem to the regularized system \eqref{rLS} in the case $Q=O$, that is, 
\begin{equation}\label{rLS2}
\begin{cases}
 \ddot{\bm{u}}=(A(s,t)\bm{u}')'+\varepsilon s\dot{\bm{u}}'+\bm{f}(s,t) &\mbox{in}\quad (0,1)\times(0,T), \\
 \bm{u}=\bm{0} &\mbox{on}\quad \{s=1\}\times(0,T), \\
 (\bm{u},\dot{\bm{u}})|_{t=0}=(\bm{u}_0^{\mathrm{in}},\bm{u}_1^{\mathrm{in}}) &\mbox{in}\quad (0,1),
\end{cases}
\end{equation}
with a degenerate but smooth coefficient $A$ and a regularizing parameter $\varepsilon\in\mathbb{R}$. 
The compatibility conditions for the data $(\bm{u}_0^{\mathrm{in}},\bm{u}_1^{\mathrm{in}},\bm{f})$ can be defined similarly to Definition \ref{def:CC1}.

\begin{proposition}\label{prop:exist1}
Let $m\geq2$ be an integer and assume that $A\in C^\infty(\overline{(0,1)\times(0,T)})$ is symmetric and satisfies 
$M_0^{-1}s\mathrm{Id} \leq A(s,t) \leq M_0s\mathrm{Id}$ for any $(s,t)\in(0,1)\times(0,T)$ with a positive constant $M_0$. 
Then, for any data $\bm{u}_0^{\mathrm{in}}\in X^m$, $\bm{u}_1^{\mathrm{in}}\in X^{m-1}$, $\bm{f}\in\mathscr{X}_T^{m-2}$ satisfying 
$\dt^{m-1}\bm{f}\in L^1(0,T;L^2)$ and the compatibility conditions up to order $m-1$, there exists a unique solution 
$\bm{u}\in\mathscr{X}_T^m$ to the initial boundary value problem \eqref{rLS2}. 
If, in addition, $\varepsilon>0$, then the solution satisfies $\bm{u}'|_{s=1}\in H^{m-1}(0,T)$. 
\end{proposition}

\begin{proof}
This proposition can be proved along with the proof of Takayama \cite[Theorem 2.1]{Takayama2018} as follows. 
Let $\bm{u}$ be a solution to \eqref{rLS2} and put $\bm{U}(x,t)=\bm{u}^\sharp(x,t)=\bm{u}(x_1^2+x_2^2,t)$ for $(x,t)\in D\times(0,T)$. 
Then, the problem \eqref{rLS2} is transformed into the initial boundary value problem 
\begin{equation}\label{TrLS}
\begin{cases}
 \displaystyle
 \dt^2\bm{U} = \sum_{j=1,2}\left( \frac14 \partial_{x_j}(A_0^\sharp(x,t) \partial_{x_j}\bm{U})
  + \frac12 \varepsilon x_j\partial_{x_j}\dt\bm{U} \right) + \bm{f}^\sharp(x,t) &\mbox{in}\quad D\times(0,T), \\
 \bm{U}=\bm{0} &\mbox{on}\quad \partial D\times(0,T), \\
 (\bm{U},\dt\bm{U})|_{t=0} = (\bm{U}_0^\mathrm{in}, \bm{U}_1^\mathrm{in}) &\mbox{in}\quad D,
\end{cases}
\end{equation}
where $A_0(s,t)=\frac{1}{s}A(s,t)=(\mathscr{M}A(\cdot,t))(s)$ and $\bm{U}_j^\mathrm{in}=(\bm{u}_j^\mathrm{in})^\sharp$ for $j=1,2$. 
By Lemma \ref{lem:NormEq}, we see that $\bm{U}_0^\mathrm{in}\in H^m(D)$, $\bm{U}_1^\mathrm{in}\in H^{m-1}(D)$, 
$\bm{f}^\sharp\in \bigcap_{j=1}^{m-2}C^j([0,T];H^{m-2-j}(D))$ satisfy $\dt^{m-1}\bm{f}^\sharp \in L^1(0,T;L^2(D))$ and compatibility conditions up to order $m-1$. 
Since the coefficient matrix $A_0^\sharp$ is strictly positive, it is classical to show the existence of a unique solution 
$\bm{U} \in \bigcap_{j=0}^mC^j([0,T];H^{m-j}(D))$ to \eqref{TrLS}, which is radially symmetric; 
for a general theory of initial boundary value problems of hyperbolic systems, see, for example, 
Benzoni and Serre \cite[Chapter 9]{BenzoniSerre2007}, M\'etivier \cite[Chapter 2]{Metivier2004}, 
Rauch and Massey \cite[Theorem 3.1]{RauchMassey1974}, and Schochet \cite[Theorem A1]{Schochet1986}. 
Therefore, we can defined $\bm{u}(s,t)$ by $\bm{u}^\sharp=\bm{U}$. 
Then, by Lemma \ref{lem:NormEq} we see that $\bm{u}\in\mathscr{X}_T^m$ and that $\bm{u}$ is a unique solution to \eqref{rLS2}. 
Moreover, by Proposition \ref{prop:HOEE} we have $\bm{u}'(1,\cdot)\in H^{m-1}(0,T)$ if $\varepsilon>0$. 
\end{proof}

We then consider the problem \eqref{rLS} with a localized term $Q(s,t)\bm{u}'(1,t)$. 
We still assume that the coefficient matrices $A$ and $Q$ are both smooth. 
Here, $\varepsilon \in(0,1]$ is fixed so that we denote the initial data by $(\bm{u}_0^{\mathrm{in}},\bm{u}_1^{\mathrm{in}})$.

\begin{proposition}\label{prop:exist2}
Let $m\geq2$ be an integer $\varepsilon\in(0,1]$ and assume that $A,Q\in C^\infty(\overline{(0,1)\times(0,T)})$ and that $A(s,t)$ is symmetric and satisfies 
$M_0^{-1}s\mathrm{Id} \leq A(s,t) \leq M_0s\mathrm{Id}$ for any $(s,t)\in(0,1)\times(0,T)$ with a positive constant $M_0$. 
Then, for any data $\bm{u}_0^{\mathrm{in}}\in X^m$, $\bm{u}_1^{\mathrm{in}}\in X^{m-1}$, $\bm{f}\in\mathscr{X}_T^{m-2}$ satisfying 
$\dt^{m-1}\bm{f}\in L^1(0,T;L^2)$ and the compatibility conditions up to order $m-1$, there exists a unique solution 
$\bm{u}\in\mathscr{X}_T^m$ to the initial boundary value problem \eqref{rLS} satisfying $\bm{u}'|_{s=1}\in H^{m-1}(0,T)$. 
\end{proposition}

\begin{proof}
We first consider the case where the data satisfy additional regularities 
$\bm{u}_0^{\mathrm{in}}\in X^{m+1}$, $\bm{u}_1^{\mathrm{in}}\in X^m$, and $\bm{f}\in\mathscr{X}_T^{m-1}$. 
Let $\bm{u}$ be a smooth solution to \eqref{rLS} and put $\bm{u}_j^\mathrm{in}=(\dt^j\bm{u})|_{t=0}$ for $j=0,1,\ldots,m+1$. 
Then, $\{\bm{u}_j^\mathrm{in}\}_{j=0}^{m+1}$ are calculated from the data by a similar recurrence formula to \eqref{HOIV} and satisfy 
$\bm{u}_j^\mathrm{in} \in X^{m+1-j}$ for $j=0,1,\ldots,m+1$. 
Therefore, we can construct $\bm{u}^{(0)} \in \mathscr{X}_T^{m+1}$ which satisfies $(\dt^j\bm{u}^{(0)})|_{t=0}=\bm{u}_j^\mathrm{in}$ for $j=0,1,\ldots,m+1$. 
Particularly, we have $(\ds\bm{u}^{(0)})|_{s=1}\in H^{m-1}(0,T)$. 
We proceed to construct a sequence of approximate solutions $\{\bm{u}^{(n)}\}_{n=0}^\infty$. 
Suppose that $\bm{u}^{(n)}\in\mathscr{X}_T^m$ is given so that 
\begin{equation}\label{CondSA}
\begin{cases}
 (\dt^j\bm{u}^{(n)})|_{t=0}=\bm{u}_j^\mathrm{in} \quad\mbox{for}\quad j=0,1,\ldots,m, \\
 (\ds\bm{u}^{(n)})|_{s=1}\in H^{m-1}(0,T),
\end{cases}
\end{equation}
and consider the initial boundary value problem 
\[
\begin{cases}
 \ddot{\bm{v}}=(A(s,t)\bm{v}')'+\varepsilon s\dot{\bm{v}}'+\bm{f}^{(n)}(s,t) &\mbox{in}\quad (0,1)\times(0,T), \\
 \bm{v}=\bm{0} &\mbox{on}\quad \{s=1\}\times(0,T), \\
 (\bm{v},\dot{\bm{v}})|_{t=0}=(\bm{u}_0^{\mathrm{in}},\bm{u}_1^{\mathrm{in}}) &\mbox{in}\quad (0,1),
\end{cases}
\]
where $\bm{f}^{(n)}=Q(\ds\bm{u}^{(n)}|_{s=1})+\bm{f}$. 
It is easy to see that $\bm{f}^{(n)}\in\mathscr{X}_T^{m-2}$, $\dt^{m-1}\bm{f}^{(n)}\in L^1(0,T;L^2)$, 
and that the data $(\bm{u}_0^{\mathrm{in}},\bm{u}_1^{\mathrm{in}},\bm{f}^{(n)})$ satisfy the compatibility conditions up to order $m-1$. 
Therefore, by Proposition \ref{prop:exist1} the above problem has a unique solution $\bm{v}\in\mathscr{X}_T^m$ satisfying \eqref{CondSA}. 
Denoting this solution by $\bm{u}^{(n+1)}$, we have constructed the approximate solutions $\{\bm{u}^{(n)}\}_{n=0}^\infty$. 
In order to see a convergence of these approximate solutions, we put $\bm{v}^{(n)}=\bm{u}^{(n+1)}-\bm{u}^{(n)}$, which solves 
\[
\begin{cases}
 \ddot{\bm{v}}^{(n+1)}=(A(s,t)\bm{v}^{(n+1)\prime})'+\varepsilon s\dot{\bm{v}}^{(n+1)\prime}+Q(s,t)\bm{v}^{(n)\prime}(1,t) &\mbox{in}\quad (0,1)\times(0,T), \\
 \bm{v}^{(n+1)}=\bm{0} &\mbox{on}\quad \{s=1\}\times(0,T), \\
 (\bm{v}^{(n+1)},\dot{\bm{v}}^{(n+1)})|_{t=0}=(\bm{0},\bm{0}) &\mbox{in}\quad (0,1).
\end{cases}
\]
By Proposition \ref{prop:HOEE}, we see that 
\begin{align*}
& I_{\gamma,T}( \opnorm{ \bm{v}^{(n+1)}(\cdot) }_m ) + \sqrt{\varepsilon}|\bm{v}^{(n+1)\prime}(1,\cdot)|_{H_\gamma^{m-1}(0,T)} \\
&\lesssim I_{\gamma,T}( \opnorm{ Q(\bm{v}^{(n)\prime}|_{s=1}) }_{m-2} ) + S_{\gamma,T}^*( \|\dt^{m-1} (Q(\bm{v}^{(n)\prime}|_{s=1}))\|_{L^2} ) \\
&\lesssim \sum_{j=0}^{m-2} I_{\gamma,T}( |\dt^j\bm{v}^{(n)\prime}(1,\cdot)| ) + \sum_{j=0}^{m-1} S_{\gamma,T}^*( |\dt^j\bm{v}^{(n)\prime}(1,\cdot)| ) \\
&\lesssim \gamma^{-\frac12}|\bm{v}^{(n)\prime}(1,\cdot)|_{H_\gamma^{m-1}(0,T)},
\end{align*}
where we used $I_{\gamma,t}(|u|) \leq C(|u(0)|+S_{\gamma,t}^*(|\dt u|)$ and \eqref{propSstar}; see \cite[Lemma 2.16]{IguchiLannes2021}. 
Therefore, by choosing $\gamma$ so large that $\gamma^{-\frac12}\ll\sqrt{\varepsilon}$ we obtain 
\[
I_{\gamma,T}( \opnorm{ \bm{v}^{(n+1)}(\cdot) }_m ) + \sqrt{\varepsilon}|\bm{v}^{(n+1)\prime}(1,\cdot)|_{H_\gamma^{m-1}(0,T)}
\leq \frac12\sqrt{\varepsilon}|\bm{v}^{(n)\prime}(1,\cdot)|_{H_\gamma^{m-1}(0,T)}
\]
for any $n=0,1,\ldots$. 
This ensures that $\{\bm{u}^{(n)}\}_{n=0}^\infty$ and $\{\bm{u}^{(n)\prime}|_{s=1}\}_{n=0}^\infty$ converge in $\mathscr{X}_T^m$ and $H^{m-1}(0,T)$, 
respectively, so that the limit $\bm{u}$ is the desired solution.

We then consider the case without any additional regularities on the data. 
By using the method in \cite{RauchMassey1974} we can construct a sequence of approximate data 
$\{(\bm{u}_0^{\mathrm{in}(n)},\bm{u}_1^{\mathrm{in}(n)},\bm{f}^{(n)})\}_{n=1}^\infty$, which satisfy the additional regularities 
$\bm{u}_0^{\mathrm{in}(n)} \in X^{m+1}$, $\bm{u}_1^{\mathrm{in}(n)} \in X^m$, $\bm{f}^{(n)} \in \mathscr{X}_T^{m-1}$, 
and compatibility conditions up to order $m-1$, and converge to the original data $(\bm{u}_0^{\mathrm{in}},\bm{u}_1^{\mathrm{in}},\bm{f})$ 
in the corresponding spaces stated in the proposition. 
Then, for each $n\in\mathbb{N}$ there exists a unique solution $\bm{u}^{(n)} \in \mathscr{X}_T^m$ to the problem corresponding to the approximate data. 
By the linearity of the problem and by Proposition \ref{prop:HOEE}, 
we see that $\{\bm{u}^{(n)}\}_{n=0}^\infty$ and $\{\bm{u}^{(n)\prime}|_{s=1}\}_{n=0}^\infty$ converge in $\mathscr{X}_T^m$ and $H^{m-1}(0,T)$, respectively, 
so that the limit $\bm{u}$ is the desired solution. 
\end{proof}

We are ready to prove one of our main result in this paper, that is, Theorem \ref{Th1}.

\begin{proof}[Proof of Theorem \ref{Th1}]
Once a solution $\bm{u}\in\mathscr{X}_T^m$ to the problem \eqref{LS} is obtained, the energy estimate \eqref{EE1} follows from Proposition \ref{prop:HOEE}. 
Since Assumption \ref{ass:HOEE} (iii) is just imposed to exhibit how the constants $C_0$ and $\gamma_1$ in \eqref{EE1} depends on norms of the coefficients $A$ and $Q$, 
it is sufficient to show the existence of a solution $\bm{u}\in\mathscr{X}_T^m$ under Assumptions \ref{ass:BEE} and \ref{ass:HOEE} (i)--(ii). 
The proof consists of 4 steps and proceeds in a similar way as the proof of \cite[Theorem A1]{Schochet1986}.

\medskip
\noindent
\textbf{Step 1.} 
We assume additionally that $A,Q\in C^\infty(\overline{(0,1)\times(0,T)})$ and that the data satisfy additional regularities 
$\bm{u}_0^{\mathrm{in}}\in X^{m+1}$, $\bm{u}_1^{\mathrm{in}}\in X^m$, $\bm{f}\in\mathscr{X}_T^m$, 
and compatibility conditions up to order $m$ to the problem \eqref{LS}. 
Let $0<\varepsilon\leq1$ and consider the regularized problem \eqref{rLS}. 
We note that the data $(\bm{u}_0^{\mathrm{in}},\bm{u}_1^{\mathrm{in}},\bm{f})$ do not necessarily satisfy the compatibility conditions 
to the regularized problem \eqref{rLS}. 
However, by using the method in \cite{RauchMassey1974} we can construct initial data 
$(\bm{u}_0^{\mathrm{in},\varepsilon},\bm{u}_1^{\mathrm{in},\varepsilon}) \in X^{m+1}\times X^m$ of the problem \eqref{rLS} so that 
the modified data $(\bm{u}_0^{\mathrm{in},\varepsilon},\bm{u}_1^{\mathrm{in},\varepsilon},\bm{f})$ satisfy compatibility conditions up to order $m$ 
and that the modified initial data converge to the original ones in $X^{m+1}\times X^m$ as $\varepsilon\to +0$. 
Then, by Proposition \ref{prop:exist2} there exists a unique solution $\bm{u}^\varepsilon \in \mathscr{X}_T^{m+1}$ of the regularized problem \eqref{rLS}. 
Moreover, by Proposition \ref{prop:HOEE} the solutions $\{\bm{u}^\varepsilon\}_{0<\varepsilon\leq1}$ satisfy the uniform bound 
$\opnorm{ \bm{u}^\varepsilon(t) }_m \leq C$ for any $t\in[0,T]$ and $\varepsilon\in(0,1]$ with a constant $C$ independent of $t$ and $\varepsilon$. 
In order to see the convergence of these solutions as $\varepsilon\to+0$, we put $\bm{v}^{\varepsilon,\eta}=\bm{u}^\varepsilon-\bm{u}^\eta$, 
which solves 
\[
\begin{cases}
 \ddot{\bm{v}}^{\varepsilon,\eta}=(A(s,t)\bm{v}^{\varepsilon,\eta \prime})'+Q(s,t)\bm{v}^{\varepsilon,\eta \prime}(1,t)
  + \bm{f}^{\varepsilon,\eta}(s,t) &\mbox{in}\quad (0,1)\times(0,T), \\
 \bm{v}^{\varepsilon,\eta}=\bm{0} &\mbox{on}\quad \{s=1\}\times(0,T), \\
 (\bm{v}^{\varepsilon,\eta},\dot{\bm{v}}^{\varepsilon,\eta})|_{t=0}
  =(\bm{u}_0^{\mathrm{in},\varepsilon}-\bm{u}_0^{\mathrm{in},\eta},\bm{u}_1^{\mathrm{in},\varepsilon}-\bm{u}_1^{\mathrm{in},\eta}) &\mbox{in}\quad (0,1),
\end{cases}
\]
where $\bm{f}^{\varepsilon,\eta} = \varepsilon s\dot{\bm{u}}^{\varepsilon \prime} - \eta s\dot{\bm{u}}^{\eta \prime}$. 
Therefore, by Proposition \ref{prop:HOEE} and Lemma \ref{lem:embedding2} we obtain 
\begin{align*}
I_{\gamma,T}( \opnorm{ (\bm{u}^\varepsilon-\bm{u}^\eta)(\cdot) }_m )
&\lesssim \|\bm{u}_0^{\mathrm{in},\varepsilon}-\bm{u}_0^{\mathrm{in},\eta}\|_{X^m} + \|\bm{u}_1^{\mathrm{in},\varepsilon}-\bm{u}_1^{\mathrm{in},\eta}\|_{X^{m-1}}
 + I_{\gamma,T}( \opnorm{ \bm{f}^{\varepsilon,\eta}(\cdot) }_{m-1} ) \\
&\lesssim \|\bm{u}_0^{\mathrm{in},\varepsilon}-\bm{u}_0^{\mathrm{in},\eta}\|_{X^m} + \|\bm{u}_1^{\mathrm{in},\varepsilon}-\bm{u}_1^{\mathrm{in},\eta}\|_{X^{m-1}} \\
&\quad\;
 + \varepsilon I_{\gamma,T}( \opnorm{ \bm{u}^\varepsilon(\cdot) }_m ) + \eta I_{\gamma,T}( \opnorm{ \bm{u}^\eta(\cdot) }_m ) \\
&\to 0 \quad\mbox{as}\quad \varepsilon,\eta\to+0,
\end{align*}
which shows that $\{\bm{u}^\varepsilon\}_{0<\varepsilon\leq1}$ 
converges in $\mathscr{X}_T^m$ and the limit $\bm{u}$ is the desired solution.

\medskip
\noindent
\textbf{Step 2.} 
We still assume that $A,Q\in C^\infty(\overline{(0,1)\times(0,T)})$ but do not assume any additional regularities on the data 
$(\bm{u}_0^\mathrm{in},\bm{u}_1^\mathrm{in},\bm{f})$. 
Then, as before we can construct a sequence of regular approximate data $\{(\bm{u}_0^{\mathrm{in}(n)},\bm{u}_1^{\mathrm{in}(n)},\bm{f}^{(n)})\}_{n=1}^\infty$, 
which satisfies the additional regularities stated in Step 1 
and converges to the original data in the corresponding spaces. 
Then, by the result in Step 1, for each $n\in\mathbb{N}$ there exists a unique solution $\bm{u}^{(n)}\in\mathscr{X}_T^m$ 
to the problem corresponding to the approximate data. 
By the linearity of the problem and by Proposition \ref{prop:HOEE}, 
we see that $\{\bm{u}^{(n)}\}_{n=0}^\infty$ converges in $\mathscr{X}_T^m$, so that the limit $\bm{u}$ is the desired solution.

\medskip
\noindent
\textbf{Step 3.} 
We will prove Theorem \ref{Th1} in the case $m\geq3$ without any additional regularities on the coefficients and the data. 
We first approximate the coefficient matrices $A$ and $Q$ by smooth ones $\{A^{(n)}\}_{n=1}^\infty$ and $\{Q^{(n)}\}_{n=1}^\infty$, which satisfy 
$A^{(n)},Q^{(n)}\in C^\infty(\overline{(0,1)\times(0,T)})$ and conditions in Assumptions \ref{ass:BEE} and \ref{ass:HOEE} 
with $M_0$ and $M_1$ replaced by $2M_0$ and $2M_1$, respectively. 
Moreover, $\{A^{(n)\prime}\}_{n=1}^\infty$ and $\{Q^{(n)}\}_{n=1}^\infty$ converge to $A'$ and $Q$ 
in $\mathscr{X}_T^{m-2} \cap \mathscr{X}_T^{2,*}$ and $\mathscr{X}_T^{m-2}$, respectively. 
We then consider the initial boundary value problem 
\begin{equation}\label{LSapp}
\begin{cases}
 \ddot{\bm{u}}=(A^{(n)}(s,t)\bm{u}')'+Q^{(n)}(s,t)\bm{u}'(1,t)+\bm{f}(s,t) &\mbox{in}\quad (0,1)\times(0,T), \\
 \bm{u}=\bm{0} &\mbox{on}\quad \{s=1\}\times(0,T), \\
 (\bm{u},\dot{\bm{u}})|_{t=0}=(\bm{u}_0^{\mathrm{in}(n)},\bm{u}_1^{\mathrm{in}(n)}) &\mbox{in}\quad (0,1),
\end{cases}
\end{equation}
where the initial data $(\bm{u}_0^{\mathrm{in}(n)},\bm{u}_1^{\mathrm{in}(n)}) \in X^m\times X^{m-1}$ can be constructed so that the data 
$(\bm{u}_0^{\mathrm{in}(n)},\bm{u}_1^{\mathrm{in}(n)}, \bm{f})$ for the above problem satisfy the compatibility conditions up to order $m-1$ and 
converge to $(\bm{u}_0^{\mathrm{in}},\bm{u}_1^{\mathrm{in}})$ in $X^m\times X^{m-1}$ as $n\to\infty$. 
Then, by the result in Step 2, for each $n\in\mathbb{N}$ the above problem has a unique solution $\bm{u}^{(n)}\in\mathscr{X}_T^m$. 
Moreover, by Proposition \ref{prop:HOEE} these solutions satisfy the uniform bound $\opnorm{ \bm{u}^{(n)}(t) }_m \leq C$ for any $t\in[0,T]$ and $n\in\mathbb{N}$ 
with a constant $C$ independent of $t$ and $n$. 
On the other hand, by Lemma \ref{lem:NormEq} we see that the embedding $X^{j+1} \hookrightarrow X^j$ is compact so that by the Aubin--Lions lemma 
the embedding $\mathscr{X}_T^m \hookrightarrow \mathscr{X}_T^{m-1}$ is also compact. 
Therefore, $\{\bm{u}^{(n)}\}_{n=1}^\infty$ has a subsequence which converges $\bm{u}$ in $\mathscr{X}_T^{m-1}$. 
Obviously, $\bm{u}$ is a unique solution to \eqref{LS}; we note here that this is the only place where the case $m=2$ is excluded. 
As a result, without taking a subsequence, $\{\bm{u}^{(n)}\}_{n=1}^\infty$ itself converges $\bm{u}$ in $\mathscr{X}_T^{m-1}$. 
Moreover, by standard compactness arguments we have also 
\[
\dt^j\bm{u}\in L^\infty(0,T;X^{m-j}) \cap C_\mathrm{w}([0,T];X^{m-j})
\]
for $j=0,1,\ldots,m$. 
It remains to show that this weak continuity in time can be replaced by the strong continuity. 
To this end, we use the technique used by Majda \cite[Chapter 2.1]{Majda1984} and Majda and Bertozzi \cite[Chapter 3.2]{MajdaBertozzi2002}, 
that is, we make use of the energy estimate. 
For the approximate solution $\bm{u}^{(n)}$, we define an energy functional $\mathscr{E}_m^{(n)}(t)$ by 
\begin{align*}
\mathscr{E}_m^{(n)}(t)
&= (A^{(n)}\dt^{m-1}\bm{u}^{(n)\prime},\dt^{m-1}\bm{u}^{(n)\prime})_{L^2} + \|(A^{(n)}\dt^{m-2}\bm{u}^{(n)\prime})'\|_{L^2}^2 \\
&\quad\;
 + 2((A^{(n)}\dt^{m-2}\bm{u}^{(n)\prime})',sQ^{(n)}\dt^{m-2}\bm{u}^{(n)\prime}+\bm{f}_m^{(n)})_{L^2},
\end{align*}
where $\bm{f}_m^{(n)}=\dt^{m-2}\bm{f}+([\dt^{m-2},A^{(n)}]\bm{u}^{(n)\prime})'+[\dt^{m-2},Q^{(n)}]\bm{u}^{(n)\prime}(1,t)$. 
Then, as in the proof of Propositions \ref{prop:BEE} and \ref{prop:HOEE} we obtain 
$\mathscr{E}_m^{(n)}(t)=\mathscr{E}_m^{(n)}(t_0)+\int_{t_0}^t F_m^{(n)}(t')\mathrm{d}t'$, 
where $F_m^{(n)}$ satisfies $|F_m^{(n)}(t)| \leq C(\|(\dt^{m-1}\bm{f},\dt^{m-2}\bm{f})(t)\|_{L^2} + 1)$ with a constant $C$ independent of $n$ and $t$. 
Passing to the limit $n\to\infty$ to this energy identity, we see that the corresponding energy functional $\mathscr{E}_m(t)$ for the solution $\bm{u}$ 
is continuous in $t$. 
This fact together with the weak continuity implies that $\dt^{m-1}\bm{u}\in C([0,T];X^1)$ and $\dt^{m-2}\bm{u}\in C([0,T];X^2)$. 
Then, by using the hyperbolic system we obtain $\dt^m\bm{u} \in C([0,T];L^2)$. 
Finally, as in the proof of Lemma \ref{lem:Estu} we can show $\dt^{m-j}\bm{u}\in C([0,T];X^j)$ inductively on $j=3,4,\ldots,m$.

\medskip
\noindent
\textbf{Step 4.} 
We will prove Theorem \ref{Th1} in the case $m=2$. 
We first note that the conditions in Assumptions \ref{ass:BEE} and \ref{ass:HOEE} (i)--(ii) in the cases $m=2$ and $m=3$ are exactly the same. 
As before, we approximate the data $(\bm{u}_0^\mathrm{in}, \bm{u}_1^\mathrm{in},\bm{f})$ by a sequence of more regular data 
$\{(\bm{u}_0^{\mathrm{in}(n)},\bm{u}_1^{\mathrm{in}(n)},\bm{f}^{(n)})\}_{n=1}^\infty$, 
which satisfies the conditions in the case $m=3$ 
and converges to the original data in the corresponding spaces. 
Then, by the result in Step 3, for each $n\in\mathbb{N}$ there exists a unique solution $\bm{u}^{(n)}\in\mathscr{X}_T^3$ 
to the problem corresponding to the approximate data. 
By the linearity of the problem and by Proposition \ref{prop:BEE}, 
we see that $\{\bm{u}^{(n)}\}_{n=0}^\infty$ converges in $\mathscr{X}_T^2$, so that the limit $\bm{u}$ is the desired solution. 
\end{proof}

\section{Two-point boundary value problem}\label{sect:TBVP}
We proceed to consider the linearized system \eqref{LEq} and \eqref{LBVP} for the motion of an inextensible hanging string. 
The solution $\nu$ of the two-point boundary value problem \eqref{LBVP} can be decomposed as a sum of a principal part $\nu_\mathrm{p}$ 
and a lower order part $\nu_\mathrm{l}$. 
The principal part $\nu_\mathrm{p}$ can be written explicitly as \eqref{pnu}, so that the lower order part $\nu_\mathrm{l}$ satisfies 
\begin{equation}\label{lnu}
\begin{cases}
 -\nu_\mathrm{l}''+|\bm{x}''|^2\nu_\mathrm{l} = 2\dot{\bm{x}}'\cdot\dot{\bm{y}}' - 2(\bm{x}''\cdot\bm{y}'')\tau + h &\mbox{in}\quad (0,1)\times(0,T), \\
 \nu_\mathrm{l} = 0 &\mbox{on}\quad \{s=0\}\times(0,T), \\
 \nu_\mathrm{l}' = -2\dot{\bm{x}}'\cdot\dot{\bm{y}} + 2(\bm{x}''\cdot\bm{y}')\tau &\mbox{on}\quad \{s=1\}\times(0,T).
\end{cases}
\end{equation}
Note that from \eqref{pnu} and \eqref{defphi}, 
the boundary condition of $\nu_\mathrm{l}$ on $\{s=1\}\times(0,T)$ is naturally $\nu_\mathrm{l}' = 2(\bm{x}''\cdot\bm{y}')\tau$. 
However, this boundary condition can be written as the last boundary condition in \eqref{lnu}, since $\bm{y}=\bm{0}$ on $\{s=1\}\times(0,T)$, 
which comes from \eqref{LEq}. 
Here, we adopt \eqref{lnu} to facilitate later analysis. 
In view of \eqref{lnu} and \eqref{defphi}, we first consider the two-point boundary value problem 
\begin{equation}\label{TBVP}
\begin{cases}
 -\nu''+|\bm{x}''|^2\nu = h \quad\mbox{in}\quad (0,1), \\
 \nu(0)=0, \quad \nu'(1)=a,
\end{cases}
\end{equation}
where $h$ is a given function and $a$ is a constant.

\begin{lemma}[{\cite[Lemma 3.7]{IguchiTakayama2023}}]\label{lem:EstSolBVP3}
For any $M>0$ there exists a constant $C=C(M)>0$ such that if $\|s^{\frac12}\bm{x}''\|_{L^2} \leq M$, then the solution $\nu$ 
to the boundary value problem \eqref{TBVP} satisfies 
\[
\|s^\alpha\nu'\|_{L^p} \leq C(|a|+\|s^{\alpha+\frac{1}{p}}h\|_{L^1})
\]
for any $p\in[1,\infty]$ and any $\alpha\geq0$ satisfying $\alpha+\frac{1}{p}\leq1$. 
\end{lemma}

The estimate in this lemma is not sufficient to guarantee that the solution $\nu_\mathrm{l}$ of \eqref{lnu} is in fact a lower order term. 
In order to show that $\nu_\mathrm{l}$ is of lower order, we need to consider the two-point boundary value problem 
\begin{equation}\label{TBVP2}
\begin{cases}
 -\nu''+|\bm{x}''|^2\nu = h_I-h_{II}' \quad\mbox{in}\quad (0,1), \\
 \nu(0)=0, \quad \nu'(1)=a+h_{II}(1),
\end{cases}
\end{equation}
where $h_I$ and $h_{II}$ are given functions and $a$ is a constant.

\begin{lemma}\label{lem:EstNu}
For any $M>0$ there exists a constant $C=C(M)>0$ such that if $\|s^{\frac12}\bm{x}''\|_{L^2} \leq M$, 
then the solution $\nu$ to the boundary value problem \eqref{TBVP2} satisfies 
\[
\|\nu'\|_{L^2} \leq C(|a|+\|s^\frac{1}{2}h_I\|_{L^1})+\|h_{II}\|_{L^2}.
\]
\end{lemma}

\begin{proof}
The estimate in the case $h_{II}=0$ comes from Lemma \ref{lem:EstSolBVP3}. 
Therefore, by the linearity of the problem, it is sufficient to show the estimate in the case $h_I=0$ and $a=0$. 
Multiplying the first equation in \eqref{TBVP2} by $\nu$ and integrating it over $[0,1]$, we see that 
\begin{align*}
\int_0^1(|\nu'(s)|^2+|\bm{x}''(s)|^2|\nu(s)|^2)\mathrm{d}s
&= \nu'(1)\nu(1) - \nu'(0)\nu(0) - \int_0^1h_{II}'(s)\nu(s)\mathrm{d}s \\
&= \int_0^1h_{II}(s)\nu'(s)\mathrm{d}s,
\end{align*}
where we used the boundary conditions. 
This implies $\|\nu'\|_{L^2}\leq\|h_{II}\|_{L^2}$. 
\end{proof}

\begin{lemma}\label{lem:EstPhi}
Let $j$ be a positive integer and $M>0$. 
There exists a constant $C=C(j,M)>0$ such that if $\bm{x}$ satisfies
  \setlength{\parskip}{-1mm}
\begin{enumerate}
  \setlength{\itemsep}{-0.5mm}
\item[{\rm (i)}]
$\opnorm{\bm{x}(t)}_{3,1} \leq M$ in the case $j=1$;
\item[{\rm (ii)}]
$\opnorm{\bm{x}(t)}_{4,2} \leq M$ in the case $j=2$;
\item[{\rm (iii)}]
$\opnorm{\bm{x}(t)}_{j+1,j-1}, \|\dt^j\bm{x}(t)\|_{X^2} \leq M$ in the case $j\geq3$,
\end{enumerate}
then the solution $\phi$ of \eqref{defphi} satisfies $\opnorm{\phi'(t)}_j \leq C$. 
\end{lemma}

\begin{proof}
We first consider the case $j=1$. 
By Lemma \ref{lem:EstSolBVP3}, under the condition $\|\bm{x}(t)\|_{X^3}\lesssim1$ we have $\|\phi'(t)\|_{L^\infty}\lesssim1$ so that $|\phi(s,t)|\lesssim s$. 
We note that $\opnorm{\phi'}_1 \leq \|(\phi',\dot{\phi}')\|_{L^2}+\|\phi''\|_{Y^0}$. 
By using the first equation in \eqref{defphi} and Lemma \ref{lem:CalIneq2}, the second term in the right-hand side can be evaluated as 
$\|\phi''\|_{Y^0}=\|\phi\bm{x}''\cdot\bm{x}''\|_{Y^0} \lesssim \|\phi'\|_{L^\infty}\|\bm{x}\|_{X^3}^2$. 
To evaluate $\|\dot{\phi}'\|_{L^2}$ we differentiate \eqref{defphi} with respect to $t$ and obtain 
\[
\begin{cases}
 -\dot{\phi}''+|\bm{x}''|^2\dot{\phi} = -2\phi\bm{x}''\cdot\dot{\bm{x}}'' &\mbox{in}\quad (0,1), \\
 \dot{\phi}(0,t)=\dot{\phi}'(1,t)=0.
\end{cases}
\]
Therefore, by Lemma \ref{lem:EstSolBVP3} we get 
\begin{align*}
\|\dot{\phi}'\|_{L^2}
&\lesssim \|s^\frac12\phi\bm{x}''\cdot\dot{\bm{x}}''\|_{L^1} \\
&\leq \|\phi'\|_{L^\infty}\|s^\frac12\bm{x}''\|_{L^2} \|s\dot{\bm{x}}\|_{L^2} \\
&\leq \|\phi'\|_{L^\infty}\|\bm{x}\|_{X^3} \|\dot{\bm{x}}\|_{X^2}.
\end{align*}
These estimates give $\opnorm{\phi'(t)}_1\lesssim1$.

We then consider the case $j\geq2$. 
Let $k$ be an integer such that $2\leq k\leq j$. 
We note that $\opnorm{\phi'}_k \leq \sum_{l=0}^k\|\dt^l\phi'\|_{L^2}+\opnorm{\phi''}_{k-1}^\dag$. 
By Lemma \ref{lem:CalIneqY2}, the second term in the right-hand side can be evaluated as 
\begin{align*}
\opnorm{\phi''}_{k-1}^\dag
&= \opnorm{\phi\bm{x}''\cdot\bm{x}''}_{k-1}^\dag \\
&\lesssim
\begin{cases}
 \opnorm{\phi'}_{1}\opnorm{\bm{x}}_{4,1}^2 &\mbox{for}\quad k=2, \\
 \opnorm{\phi'}_{k-1}\opnorm{\bm{x}}_{k+1,k-1}^2 &\mbox{for}\quad k\geq3 
\end{cases}\\
&\lesssim \opnorm{\phi'}_{k-1}.
\end{align*}
We proceed to evaluate $\|\dt^k\phi'\|_{L^2}$. 
Differentiating \eqref{defphi} $k$-times with respect to $t$ we obtain 
\[
\begin{cases}
 -(\dt^k\phi)''+|\bm{x}''|^2(\dt^k\phi) = -[\dt^k,|\bm{x}''|^2]\phi &\mbox{in}\quad (0,1), \\
 (\dt^k\phi)(0,t)=(\dt^k\phi)'(1,t)=0.
\end{cases}
\]
Therefore, by Lemma \ref{lem:EstSolBVP3} we get 
\begin{align*}
\|\dt^k\phi'\|_{L^2}
&\lesssim \|s^\frac12[\dt^k,|\bm{x}''|^2]\phi\|_{L^1} \\
&\lesssim \textstyle\sum' I(k_0,k_1,k_2), 
\end{align*}
where $I(k_0,k_1,k_2)=\|s^\frac12(\dt^{k_0}\phi)(\dt^{k_1}\bm{x})''\cdot(\dt^{k_2}\bm{x})''\|_{L^1}$ and $\sum'$ denotes the summation over all $(k_0,k_1,k_2)$ 
satisfying $k_0+k_1+k_2=k$, $k_0\leq k-1$, and $k_1\leq k_2$. 
%
By Lemma \ref{lem:CalIneqY3}, we see that 
\begin{align*}
I(k_0,k_1,k_2)
&\lesssim
\begin{cases}
 \|\phi'\|_{L^\infty} \|\dt^{k_1}\bm{x}\|_{X^3} \|\dt^{k_2}\bm{x}\|_{X^2} &\mbox{for}\quad (k_0,k_1,k_2)=(0,0,k),(0,1,k-1), \\
 \|\dot{\phi}'\|_{L^2} \|\bm{x}\|_{X^4} \|\dt^{k-1}\bm{x}\|_{X^2} &\mbox{for}\quad (k_0,k_1,k_2)=(1,0,k-1), \\
 \|\dt^{k_0}\phi'\|_{L^2} \|\dt^{k_1}\bm{x}\|_{X^3} \|\dt^{k_2}\bm{x}\|_{X^3} &\mbox{for}\quad k_2\leq k-2
\end{cases} \\
&\lesssim 1+\opnorm{\phi'}_{k-1}.
\end{align*}
Summarizing the above estimates we obtain $\opnorm{\phi'}_k \lesssim \opnorm{\phi'}_{k-1}+1$ for $2\leq k\leq j$. 
Therefore, we get $\opnorm{\phi'}_j \lesssim 1$. 
\end{proof}

\begin{lemma}\label{lem:EstDtNu1}
Let $j$ be a non-negative integer and $M>0$. 
There exists a constant $C=C(j,M)>0$ such that if $\bm{x}$ and $\tau$ satisfy $\tau(0,t)=0$, $\|\tau'(t)\|_{L^\infty}\leq M$, and 
  \setlength{\parskip}{-1mm}
\begin{enumerate}
  \setlength{\itemsep}{-0.5mm}
\item[{\rm (i)}]
$\opnorm{\bm{x}(t)}_{3,1} \leq M$ in the case $j=0$; 
\item[{\rm (ii)}]
$\opnorm{\bm{x}(t)}_{4,2},\|\dot{\bm{x}}'(t)\|_{L^\infty},\|\dot{\tau}'(t)\|_{L^2} \leq M$ in the case $j=1$; 
\item[{\rm (iii)}]
$\|\bm{x}(t)\|_{X^4}, \sum_{k=1}^{j}\|\dt^k\bm{x}(t)\|_{X^3}, \|\dt^{j+1}\bm{x}(t)\|_{X^1}, \|\dot{\bm{x}}'(t)\|_{L^\infty}, 
 \sum_{k=1}^j\|\dt^k\tau'(t)\|_{L^2} \leq M$ in the case $j\geq2$, 
\end{enumerate}
then the solution $\nu_\mathrm{l}$ to the boundary value problem \eqref{lnu} satisfies 
\[
\|\dt^j\nu_\mathrm{l}'(t)\|_{L^2} \leq
\begin{cases}
 C\bigl( \|s^\frac12 h(t)\|_{L^1} + \opnorm{\bm{y}(t)}_{2,*} \bigr) &\mbox{for}\quad j=0, \\
 C\bigl( \sum_{k=0}^j\|s^\frac12 \dt^k h(t)\|_{L^1} + \opnorm{\bm{y}(t)}_{j+1} \bigr) &\mbox{for}\quad j\geq1.
\end{cases}
\]
\end{lemma}

\begin{remark}\label{re:EstDtNu}
If we impose an additional condition $\opnorm{\bm{x}(t)}_{4,1}, \|\dot{\bm{x}}'(t)\|_{L^\infty} \leq M$ in the case $j=0$, 
then we can improve the estimate as $\|\nu_\mathrm{l}'(t)\|_{L^2} \leq C\bigl( \|s^\frac12 h(t)\|_{L^1} + \opnorm{\bm{y}(t)}_1 \bigr)$. 
\end{remark}

\begin{proof}[Proof of Lemma \ref{lem:EstDtNu1}]
We first note that $|\tau(s,t)|\leq Ms$. 
By Lemmas \ref{lem:EstNu} and \ref{lem:embedding2}, we see that 
\begin{align*}
\|\nu_\mathrm{l}'\|_{L^2}
&\lesssim |(\tau\bm{x}''\cdot\bm{y}')|_{s=1}| + \|s^\frac12(2\dot{\bm{x}}'\cdot\dot{\bm{y}}' - 2(\bm{x}''\cdot\bm{y}'')\tau + h)\|_{L^1} \\
&\lesssim \|\bm{x}\|_{X^3}\|\bm{y}\|_{X^2} + \|\dot{\bm{x}}'\|_{L^2} \|s^\frac12\dot{\bm{y}}'\|_{L^2} + \|s^\frac12\bm{x}''\|_{L^2} \|s\bm{y}''\|_{L^2}
 + \|s^\frac12 h\|_{L^1} \\
&\lesssim \|\bm{x}\|_{X^3}\|\bm{y}\|_{X^2} + \|\dot{\bm{x}}\|_{X^2} \|\dot{\bm{y}}\|_{X^1} + \|\bm{x}\|_{X^3} \|\bm{y}\|_{X^2} + \|s^\frac12 h\|_{L^1}.
\end{align*}
This gives the desired estimate in the case $j=0$. 

We then consider the case $j\geq1$. 
Let $k$ be an integer such that $1\leq k\leq j$. 
Differentiating the equations in \eqref{lnu} $k$-times with respect to $t$, we have 
\begin{equation}\label{DtBVP}
\begin{cases}
 -(\dt^k\nu_\mathrm{l})''+|\bm{x}''|^2(\dt^k\nu_\mathrm{l}) = h_{k,I}-h_{k,II}' &\mbox{in}\quad (0,1)\times(0,T), \\
 (\dt^k\nu_\mathrm{l}) = 0 &\mbox{on}\quad \{s=0\}\times(0,T), \\
 (\dt^k\nu_\mathrm{l})' = a_k+h_{k,II} &\mbox{on}\quad \{s=1\}\times(0,T),
\end{cases}
\end{equation}
where $a_k= -2([\dt^k,\dot{\bm{x}}']\cdot\dot{\bm{y}})|_{s=1} + 2([\dt^k,\tau\bm{x}'']\cdot\bm{y}')|_{s=1}$, 
$h_{k,II}=2\tau(\bm{x}''\cdot\dt^k\bm{y}')-2\dot{\bm{x}}'\cdot\dt^{k+1}\bm{y}$, and 
\begin{align*}
h_{k,I}
&= \dt^kh + 2([\dt^k,\dot{\bm{x}}']\cdot\dot{\bm{y}}'-[\dt^k,\tau\bm{x}'']\cdot\bm{y}'') - [\dt^k,|\bm{x}''|^2]\nu_\mathrm{l}
 + 2((\tau\bm{x}'')'\cdot\dt^k\bm{y}' - \dot{\bm{x}}''\cdot\dt^{k+1}\bm{y}). 
\end{align*}
Here, we see that 
\begin{align*}
|a_k| 
&\lesssim \sum_{k_1+k_2=k,k_2\leq k-1} |\dt^{k_1+1}\bm{x}'(1,t)| |\dt^{k_2+1}\bm{y}(1,t)| \\
&\quad\;
 + \sum_{k_0+k_1+k_2=k,k_2\leq k-1} |\dt^{k_0}\tau(1,t)| |\dt^{k_1}\bm{x}''(1,t)| |\dt^{k_2}\bm{y}'(1,t)| \\
&\lesssim \sum_{k_1+k_2=k,k_2\leq k-1} \|\dt^{k_1+1}\bm{x}\|_{X^2} \|\dt^{k_2+1}\bm{y}\|_{X^1}
 + \sum_{k_0+k_1+k_2=k,k_2\leq k-1} \|\dt^{k_0}\tau'\|_{L^2} \|\dt^{k_1}\bm{x}\|_{X^3} \|\dt^{k_2}\bm{y}\|_{X^2} \\
&\lesssim \opnorm{\bm{y}}_{k+1}
\end{align*}
and, by Lemma \ref{lem:CalIneqLp1}, that 
\begin{align*}
\|h_{k,II}\|_{L^2}
&\lesssim \|s^\frac12 \bm{x}''\|_{L^\infty} \|s^\frac12\dt^k\bm{y}'\|_{L^2} + \|\dot{\bm{x}}'\|_{L^\infty} \|\dt^{k+1}\bm{y}\|_{L^2} \\
&\lesssim \|\bm{x}\|_{X^4} \|\dt^k\bm{y}\|_{X^1} + \|\dot{\bm{x}}'\|_{L^\infty}\|\dt^{k+1}\bm{y}\|_{L^2} \\
&\lesssim \opnorm{\bm{y}}_{k+1}.
\end{align*}
We proceed to evaluate $\|s^\frac12 h_{k,II}\|_{L^1}$ term by term. 
We see that 
\begin{align*}
& \|s^\frac12((\tau\bm{x}'')'\cdot\dt^k\bm{y}' - \dot{\bm{x}}''\cdot\dt^{k+1}\bm{y})\|_{L^1} \\
&\lesssim \|\tau'\|_{L^\infty} ( \|s\bm{x}'''\|_{L^2}+\|\bm{x}''\|_{L^2} )\|s^\frac12\dt^k\bm{y}'\|_{L^2}
 + \|s^\frac12\dot{\bm{x}}''\|_{L^2} \|\dt^{k+1}\bm{y}\|_{L^2} \\
&\lesssim \|\bm{x}\|_{X^4} \|\dt^k\bm{y}\|_{X^1} + \|\dot{\bm{x}}\|_{X^3} \|\dt^{k+1}\bm{y}\|_{L^2}
\end{align*}
and that 
\begin{align*}
\|s^\frac12[\dt^k,\dot{\bm{x}}']\cdot\dot{\bm{y}}'\|_{L^1}
&\lesssim \sum_{k_1+k_2=k-1}\|s^\frac12(\dt^{k_1+2}\bm{x}')\cdot(\dt^{k_2+1}\bm{y}')\|_{L^1} \\
&\lesssim \|\dt^2\bm{x}'\|_{L^2} \|s^\frac12\dt^k\bm{y}'\|_{L^2} + \sum_{k_1+k_2=k-2} \|s^\frac12\dt^{k_1+3}\bm{x}'\|_{L^2} \|\dt^{k_2+1}\bm{y}'\|_{L^2} \\
&\lesssim \|\dt^2\bm{x}\|_{X^2} \|\dt^k\bm{y}\|_{X^1} + \|(\dt^3\bm{x},\ldots,\dt^{k+1}\bm{x})\|_{X^1} \|(\dt\bm{y},\ldots,\dt^{k-1}\bm{y})\|_{X^2},
\end{align*}
where we used Lemma \ref{lem:embedding2}. 
We see also that 
\[
\|s^\frac12[\dt^k,\tau\bm{x}'']\cdot\bm{y}''\|_{L^1}
\lesssim \sum_{k_0+k_1+k_2=k,k_2\leq k-1} \|s^\frac12(\dt^{k_0}\tau)(\dt^{k_1}\bm{x})''\cdot(\dt^{k_2}\bm{y})''\|_{L^1}.
\]
We evaluate $I(k_0,k_1,k_2)=\|s^\frac12(\dt^{k_0}\tau)(\dt^{k_1}\bm{x})''\cdot(\dt^{k_2}\bm{y})''\|_{L^1}$, where $k_0+k_1+k_2=k$ and $k_2\leq k-1$. 
By Lemma \ref{lem:CalIneqY3}, we see that 
\[
I(k_0,k_1,k_2) \lesssim
\begin{cases}
 \|\tau'\|_{L^\infty} \|\dt\bm{x}\|_{X^3} \|\dt^{k-1}\bm{y}\|_{X^2} &\mbox{for}\quad (k_0,k_1,k_2)=(0,1,k-1), \\
 \|\dt\tau'\|_{L^2} \|\bm{x}\|_{X^4} \|\dt^{k-1}\bm{y}\|_{X^2} &\mbox{for}\quad (k_0,k_1,k_2)=(1,0,k-1), \\
 \|\tau'\|_{L^\infty} \|\dt^2\bm{x}\|_{X^2} \|\dt^{k-2}\bm{y}\|_{X^3} &\mbox{for}\quad (k_0,k_1,k_2)=(0,2,k-2), \\
 \|\dt^{k_0}\tau'\|_{L^2} \|\dt^{k_1}\bm{x}\|_{X^3} \|\dt^{k-2}\bm{y}\|_{X^3} &\mbox{for}\quad k_1\leq 1,\ k_2=k-2, \\
 \|\dt^{k_0}\tau'\|_{L^2} \|\dt^{k_1}\bm{x}\|_{X^2} \|\dt^{k_2}\bm{y}\|_{X^4} &\mbox{for}\quad k_2\leq k-3,
\end{cases}
\]
so that $\|s^\frac12[\dt^k,\tau\bm{x}'']\cdot\bm{y}''\|_{L^1} \lesssim \opnorm{\bm{y}}_{k+1}$. 
Finally, we see that 
\begin{align*}
\|s^\frac12[\dt^k,|\bm{x}''|^2]\nu_\mathrm{l}\|_{L^1}
&\lesssim \sum_{k_0+k_1+k_2=k,k_0\leq k-1}\|s^\frac12(\dt^{k_0}\nu_\mathrm{l})(\dt^{k_1}\bm{x})''\cdot(\dt^{k_2}\bm{x})''\|_{L^1}.
\end{align*}
We evaluate $J(k_0,k_1,k_2)=\|s^\frac12(\dt^{k_0}\nu_\mathrm{l})(\dt^{k_1}\bm{x})''\cdot(\dt^{k_2}\bm{x})''\|_{L^1}$, where $k_0+k_1+k_2=k$ and $k_0\leq k-1$. 
By Lemma \ref{lem:CalIneqY3}, we see that 
\[
J(k_0,k_1,k_2) \lesssim
\begin{cases}
 \|\nu_\mathrm{l}'\|_{L^2} \|\bm{x}\|_{X^4} \|\dt^k\bm{x}\|_{X^2} &\mbox{for}\quad k_1=k \mbox{ or } k_2=k, \\
 \|\dt^{k_0}\nu_\mathrm{l}'\|_{L^2} \|\dt^{k_1}\bm{x}\|_{X^3}\|\dt^{k_2}\bm{x}\|_{X^3} &\mbox{for}\quad k_1,k_2 \leq k-1,
\end{cases}
\]
so that $\|s^\frac12[\dt^k,|\bm{x}''|^2]\nu_\mathrm{l}\|_{L^1} \lesssim \sum_{l=0}^{k-1}\|\dt^l\nu_\mathrm{l}'\|_{L^2}$. 
Summarizing the above estimates, we obtain 
\[
\|s^\frac12 h_{k,I}\|_{L^1} \lesssim \|s^\frac12\dt^k h\|_{L^1} + \sum_{l=0}^{k-1} \|\dt^l\nu_\mathrm{l}'\|_{L^2} + \opnorm{\bm{y}}_{k+1}.
\]
Now, we apply Lemma \ref{lem:EstNu} to the solution $\nu_\mathrm{l}$ of \eqref{DtBVP} and obtain 
\begin{align*}
\|\dt^k\nu_\mathrm{l}'\|_{L^2}
&\lesssim |a_k| + \|s^\frac12 h_{k,I}\|_{L^1} + \|h_{k,II}\|_{L^2} \\
&\lesssim \sum_{l=0}^{k-1} \|\dt^l\nu_\mathrm{l}'\|_{L^2} + \|s^\frac12\dt^k h\|_{L^1} + \opnorm{\bm{y}}_{k+1}.
\end{align*}
Using this inductively on $k=1,2,\ldots,j$, we obtain the desired estimate. 
\end{proof}

\begin{lemma}\label{lem:EstDtNu2}
Let $j$ be a positive integer and $M>0$. 
There exists a constant $C=C(j,M)>0$ such that if $\bm{x}$ and $\tau$ satisfy $\tau(0,t)=0$, $\|\tau'(t)\|_{L^\infty}\leq M$, and 
  \setlength{\parskip}{-1mm}
\begin{enumerate}
  \setlength{\itemsep}{-0.5mm}
\item[{\rm (i)}]
$\|(\bm{x},\dot{\bm{x}})(t)\|_{X^4}, \|\ddot{\bm{x}}(t)\|_{X^2}, \|\dot{\tau}'(t)\|_{L^2} \leq M$ in the case $j=1$;
\item[{\rm (ii)}]
$\opnorm{(\bm{x},\dot{\bm{x}})(t)}_{j+1\vee4,*}, \opnorm{\ddot{\bm{x}}(t)}_{j+1,*},\opnorm{\tau'(t)}_{j-2}, \|(\dt^{j-1}\tau',\dt^j\tau')(t)\|_{L^2} \leq M$ 
in the case $j\geq2$ and, in addition, $\|\dt^{j-3}\tau'(t)\|_{L^\infty}\leq M$ in the case $j\geq4$, 
\end{enumerate}
then the solution $\nu_\mathrm{l}$ to the boundary value problem \eqref{lnu} satisfies 
\[
\opnorm{\nu_\mathrm{l}'(t)}_j \leq C\left( \opnorm{h(t)}_{j-1}^\dag + \|s^\frac12 \dt^j h(t)\|_{L^1} + \opnorm{\bm{y}(t)}_{j+1} \right).
\]
\end{lemma}

\begin{proof}
Let $k$ be a positive integer such that $k\leq j$. 
We note that $\opnorm{\nu_\mathrm{l}'(t)}_k \leq \sum_{l=0}^k\|\dt^l\nu_\mathrm{l}'\|_{L^2} + \opnorm{\nu_\mathrm{l}''}_{k-1}^\dag$. 
By Lemma \ref{lem:EstDtNu1}, the first term in the right-hand side can be evaluated as 
\begin{align*}
\sum_{l=0}^k\|\dt^l\nu_\mathrm{l}'\|_{L^2}
&\lesssim \sum_{l=0}^j \|s^\frac12\dt^l h\|_{L^1} + \opnorm{\bm{y}}_{j+1} \\
&\lesssim \opnorm{h}_{j-1}^\dag + \|s^\frac12 \dt^j h\|_{L^1} + \opnorm{\bm{y}}_{j+1}. 
\end{align*}
By using the first equation in \eqref{lnu}, we have $\opnorm{\nu_\mathrm{l}''}_{k-1}^\dag \lesssim \opnorm{h}_{j-1}^\dag
 + \opnorm{\dot{\bm{x}}'\cdot\dot{\bm{y}}'}_{j-1}^\dag + \opnorm{\tau(\bm{x}''\cdot\bm{y}'')}_{j-1}^\dag + \opnorm{\nu_\mathrm{l}(\bm{x}''\cdot\bm{x}'')}_{k-1}^\dag$. 
By Lemmas \ref{lem:CalIneq1}--\ref{lem:CalIneqY2}, we can check that 
$\opnorm{\dot{\bm{x}}'\cdot\dot{\bm{y}}'}_{j-1}^\dag + \opnorm{\tau(\bm{x}''\cdot\bm{y}'')}_{j-1}^\dag \lesssim \opnorm{\bm{y}}_{j+1}$. 
Since $\nu_\mathrm{l}|_{s=0}=0$, we see also that 
\begin{align*}
\opnorm{\nu_\mathrm{l}(\bm{x}''\cdot\bm{x}'')}_{k-1}^\dag
&\lesssim
\begin{cases}
 \|\nu_\mathrm{l}'\|_{L^2} \|\bm{x}\|_{X^4}^2 &\mbox{for}\quad k=1, \\
 (\|\nu_\mathrm{l}'\|_{L^2} + \|\dt\nu_\mathrm{l}'\|_{L^2}) \opnorm{\bm{x}}_{4,1}^2 &\mbox{for}\quad k=2,
\end{cases}
\end{align*}
which is already evaluated. 
Therefore, we obtain the desired estimate in the case $j=1,2$. 
Moreover, for $3\leq k\leq j$ we have 
\begin{align*}
\opnorm{\nu_\mathrm{l}(\bm{x}''\cdot\bm{x}'')}_{k-1}^\dag
&\lesssim (\|\dt^{k-3}\nu_\mathrm{l}'\|_{L^\infty} + \opnorm{\nu_\mathrm{l}'}_{k-2} + \|\dt^{k-1}\nu_\mathrm{l}'\|_{L^2}) \opnorm{\bm{x}}_{k+1,k-1}^2 \\
&\lesssim \opnorm{\nu_\mathrm{l}'}_{k-1},
\end{align*}
so that 
$\opnorm{\nu_\mathrm{l}'(t)}_k \lesssim \opnorm{h(t)}_{j-1}^\dag + \|s^\frac12 \dt^j h(t)\|_{L^1} + \opnorm{\bm{y}(t)}_{j+1} + \opnorm{\nu_\mathrm{l}'(t)}_{k-1}$. 
Using this inductively on $k=3,4,\ldots,j$, we obtain the desired estimate. 
\end{proof}

\section{Estimates for initial values and compatibility conditions II}\label{sect:EstIV2}
We consider the initial boundary value problem for the linearized system \eqref{LEq} and \eqref{LBVP}. 
Let $(\bm{y},\nu)$ be a smooth solution to the problem and put $\bm{y}_j^\mathrm{in}=(\dt^j\bm{y})|_{t=0}$ and $\nu_j^\mathrm{in}=(\dt^j\nu)|_{t=0}$. 
Applying $\dt^j$ to \eqref{LEq} and \eqref{LBVP}, we see that $(\bm{y}_j^\mathrm{in},\nu_j^\mathrm{in})$ are determined inductively by 
\begin{equation}\label{RRIV1}
\bm{y}_{j+2}^\mathrm{in} = \sum_{j_0+j_1=j}\frac{j!}{j_0!j_1!}\bigl( 
 (\dt^{j_0}\tau)|_{t=0}(\bm{y}_{j_1}^{\mathrm{in}})' + \nu_{j_0}^\mathrm{in}(\dt^{j_1}\bm{x}')|_{t=0} \bigr)' + (\dt^j\bm{f})|_{t=0}
\end{equation}
and 
\begin{equation}\label{RRIV2}
\begin{cases}
 -(\nu_j^\mathrm{in})'' + |\bm{x}''|^2\nu_j^\mathrm{in} = h_j^\mathrm{in} \quad\mbox{in}\quad (0,1), \\
 \nu_j^\mathrm{in}(0)=0, \quad (\nu_j^\mathrm{in})'(1)=-\bm{g}\cdot(\bm{y}_j^\mathrm{in})'(1)
\end{cases}
\end{equation}
for $j=0,1,\ldots$, where 
\begin{align*}
h_j^\mathrm{in}
&= 2\sum_{j_1+j_2=j}\frac{j!}{j_1!j_2!}(\dt^{j_1+1}\bm{x}')|_{t=0}\cdot(\bm{y}_{j_2+1}^\mathrm{in})' 
 -2\sum_{j_0+j_1+j_2=j}\frac{j!}{j_0!j_1!j_2!}(\dt^{j_0}\tau)|_{t=0}(\dt^{j_1}\bm{x}'')|_{t=0}\cdot(\bm{y}_{j_2}^\mathrm{in})'' \\
&\quad\;
 - \sum_{j_0+j_1+j_2=j, j_0\leq j-1}\frac{j!}{j_0!j_1!j_2!}\nu_{j_0}^\mathrm{in}(\dt^{j_1}\bm{x}''\cdot\dt^{j_2}\bm{x}'')|_{t=0}
 + (\dt^jh)|_{t=0}.
\end{align*}
In fact, once the initial data $(\bm{y}_0^\mathrm{in},\bm{y}_1^\mathrm{in})$ are given, the two-point boundary value problem \eqref{RRIV2} in the case 
$j=0$ determines $\nu_0^\mathrm{in}$. 
Then, \eqref{RRIV1} with $j=0$ determines $\bm{y}_2^\mathrm{in}$. 
Then, the two-point boundary value problem \eqref{RRIV2} in the case $j=1$ determines $\nu_1^\mathrm{in}$. 
Then, \eqref{RRIV1} with $j=1$ determines $\bm{y}_3^\mathrm{in}$, and so on. 
On the other hand, by applying the boundary condition in \eqref{LEq} on $s=0$ and putting $t=0$, we obtain 
\begin{equation}\label{CC2}
\bm{y}_j^\mathrm{in}(1)=\bm{0}
\end{equation}
for $j=0,1,2,\ldots$. 
These are necessary conditions that the data $(\bm{y}_0^\mathrm{in},\bm{y}_1^\mathrm{in},\bm{f},h)$ should satisfy for the existence of a regular solution 
to the problem \eqref{LEq} and \eqref{LBVP}, and are known as compatibility conditions. 
To state the conditions more precisely, we need to evaluate the initial values $\{\bm{y}_j^\mathrm{in}\}$. 
Although it is sufficient to evaluate $\dt^j\bm{y}$ only at time $t=0$, we will evaluate them at general time $t$.

\begin{lemma}\label{lem:EstIV2}
Let $m\geq2$ be an integer and assume that Assumptions \ref{ass:xtau} and \ref{ass:addxtau} are satisfied with a positive constant $M_0$ and that 
that $\bm{f}\in\mathscr{X}_T^{m-2}$ and $s^\frac12\dt^{m-2}h\in C^0([0,T];L^1)$. 
In the case $m\geq3$, assume also that $h\in\mathscr{Y}_T^{m-3}$. 
Then, there exists a positive constant $C_0$ depending only on $m$ and $M_0$ such that if $(\bm{y},\nu)$ is a solution to \eqref{LEq} and \eqref{LBVP}, 
then we have 
\[
\opnorm{\bm{y}(t)}_m \leq C_0 \bigl( \opnorm{\bm{y}(t)}_{m,1} + \opnorm{\bm{f}(t)}_{m-2} + \opnorm{h(t)}_{m-3}^\dag + \|s^\frac12\dt^{m-2} h(t)\|_{L^1} \bigr),
\]
where we used a notational convention $\opnorm{\cdot}_{-1}^\dag=0$.
\end{lemma}

\begin{proof}
As before, we decompose the solution $\nu$ as a sum of a principal part $\nu_\mathrm{p}$ and a lower order part $\nu_\mathrm{l}$, 
where $\nu_\mathrm{p}$ is defined by \eqref{pnu} so that $\nu_\mathrm{l}$ is a unique solution to the two-point boundary value problem \eqref{lnu}. 
Then, we see that $\bm{y}$ satisfies \eqref{rLEq} with $A(s,t)=\tau(s,t)\mathrm{Id}$ and $Q(s,t)=-(\phi\bm{x}')'(s,t)\otimes(\bm{g}+2\tau\bm{x}'')(1,t)$. 
By Lemmas \ref{lem:CalIneqLp1}, \ref{lem:EstAu2}, and \ref{lem:EstPhi}, we can check easily that these matrices $A(s,t)$ and $Q(s,t)$ satisfy the conditions in 
Assumptions \ref{ass:BEE} and \ref{ass:HOEE} with the constant $M_0$ replaced by a constant $C_0=C(m,M_0)$. 
Therefore, we can apply Lemma \ref{lem:EstIV} to obtain $\opnorm{\bm{y}}_m \lesssim \opnorm{\bm{y}}_{m,1} + \opnorm{\bm{F}}_{m-2}$, 
where $\bm{F}=\bm{f}+(\nu_\mathrm{l}\bm{x}')'$. 
Here, by Lemmas \ref{lem:EstAu}, \ref{lem:EstAu2}, \ref{lem:EstDtNu1}, and \ref{lem:EstDtNu2} we see that 
\begin{align*}
\opnorm{(\nu_\mathrm{l}\bm{x}')'}_j
&\lesssim \opnorm{\nu_\mathrm{l}'}_j \\
&\lesssim
\begin{cases}
 \|s^\frac12 h\|_{L^1} + \opnorm{\bm{y}}_{2,1} &\mbox{for}\quad j=0, \\
 \opnorm{h}_{j-1}^\dag + \|s^\frac12\dt^j h\|_{L^1} + \opnorm{\bm{y}}_{j+1} &\mbox{for}\quad j=1,2,\ldots,m-2.
\end{cases}
\end{align*}
Using these estimates inductively on $j$, we obtain the desired estimate. 
\end{proof}

Under the same assumptions in Lemma \ref{lem:EstIV2}, we see that if the initial data satisfy $\bm{y}_j^\mathrm{in} \in X^{m-j}$ for $j=0,1$, 
then the initial values $\{\bm{y}_j^\mathrm{in}\}$ satisfy $\bm{y}_j^\mathrm{in} \in X^{m-j}$ for $j=0,1,\ldots,m$, so that 
their boundary values $\bm{y}_j^\mathrm{in}(1)$ are defined for $j=0,1,\ldots,m-1$.

\begin{definition}\label{def:CC2}
Let $m\geq1$ be an integer. 
We say that the data $(\bm{y}_0^\mathrm{in},\bm{y}_1^\mathrm{in},\bm{f},h)$ for the initial boundary value problem \eqref{LEq} and \eqref{LBVP} satisfy 
the compatibility conditions up to order $m-1$ if \eqref{CC2} holds for any $j=0,1,\ldots,m-1$. 
\end{definition}

\section{Existence of solutions II}\label{sect:Proof2}
In this last section we prove Theorem \ref{Th2}. 
In the following calculations, we simply denote by $C_0$ the constant depending only on $M_0$ and by $C_1$ the constant depending also on $M_1$. 
These constants may change from line to line. 
We assume that the data $(\bm{y}_0^\mathrm{in},\bm{y}_1^\mathrm{in},\bm{f},h)$ satisfy the conditions in Theorem \ref{Th2}, define initial values 
$\{\bm{y}_j^\mathrm{in}\}_{j=0}^m$ and $\{\nu_j^\mathrm{in}\}_{j=0}^{m-2}$ by \eqref{RRIV1} and \eqref{RRIV2}, and put 
\[
\mathscr{S}_T^m = \{ \bm{y} \in \mathscr{X}_T^m \,|\, (\dt^j\bm{y})|_{t=0}=\bm{y}_j^\mathrm{in} \mbox{ for }j=0,1,\ldots,m \}.
\]
By Lemma \ref{lem:EstIV2}, we have $\bm{y}_j^\mathrm{in} \in X^{m-j}$ for $j=0,1,\ldots,m$ so that it is standard to show $\mathscr{S}_T^m\ne\emptyset$. 
We take $\bm{y}^{(1)} \in \mathscr{S}_T^m$ arbitrarily and fix it. 
Given $\bm{y}^{(n)} \in  \mathscr{S}_T^m$, let $\nu_\mathrm{l}^{(n)}$ be a unique solution to the two-point boundary value problem 
\[
\begin{cases}
 -\nu_\mathrm{l}^{(n)\prime\prime}+|\bm{x}''|^2\nu_\mathrm{l}^{(n)}
  = 2\dot{\bm{x}}'\cdot\dot{\bm{y}}^{(n)\prime} - 2(\bm{x}''\cdot\bm{y}^{(n)\prime\prime})\tau + h &\mbox{in}\quad (0,1)\times(0,T), \\
 \nu_\mathrm{l}^{(n)} = 0 &\mbox{on}\quad \{s=0\}\times(0,T), \\
 \nu_\mathrm{l}^{(n)\prime} = -2\dot{\bm{x}}'\cdot\dot{\bm{y}}^{(n)} + 2(\bm{x}''\cdot\bm{y}^{(n)\prime})\tau &\mbox{on}\quad \{s=1\}\times(0,T).
\end{cases}
\]
By Lemmas \ref{lem:EstDtNu1}, \ref{lem:EstDtNu2}, \ref{lem:EstAu}, and \ref{lem:EstAu2}, we have 
$(\nu_\mathrm{l}^{(n)}\bm{x}')' \in \mathscr{X}_T^{m-2}$ and 
\begin{align*}
& \opnorm{(\nu_\mathrm{l}^{(n)}\bm{x}')'}_{m-2} \leq
 \begin{cases}
  C_0( \|s^\frac12 h\|_{L^1} + \opnorm{\bm{y}^{(n)}}_{2,1} ) &\mbox{for}\quad m=2, \\
  C_0( \opnorm{h}_{m-3}^\dag + \|s^\frac12\dt^{m-2}h\|_{L^1} + \opnorm{\bm{y}^{(n)}}_{m-1} ) &\mbox{for}\quad m\geq3, 
 \end{cases} \\
& \|\dt^{m-1}(\nu_\mathrm{l}^{(n)}\bm{x}')'\|_{L^2} \leq C_1\Biggl( \opnorm{h}_{m-3}^\dag + \sum_{j=0}^{m-1} \|s^\frac12\dt^j h\|_{L^1} + \opnorm{\bm{y}^{(n)}}_m \Biggr).
\end{align*}
Then, we consider the initial boundary value problem 
\begin{equation}\label{nthLS}
\begin{cases}
 \ddot{\bm{y}}=(A\bm{y}')'+Q\bm{y}'(1,t)+\bm{f}^{(n)} &\mbox{in}\quad (0,1)\times(0,T), \\
 \bm{y}=\bm{0} &\mbox{on}\quad \{s=1\}\times(0,T), \\
 (\bm{y},\dot{\bm{y}})|_{t=0}=(\bm{y}_0^\mathrm{in},\bm{y}_1^\mathrm{in}) &\mbox{in}\quad (0,1),
\end{cases}
\end{equation}
where $\bm{f}^{(n)}=\bm{f}+(\nu_\mathrm{l}^{(n)}\bm{x}')'$, $A(s,t)=\tau(s,t)\mathrm{Id}$, and $Q(s,t)=-(\phi\bm{x}')'(s,t)\otimes(\bm{g}+2\tau\bm{x}'')(1,t)$, 
and these matrices satisfy the conditions in Assumptions \ref{ass:BEE} and \ref{ass:HOEE} with the constants $M_0$ and $M_1$ replaced by $C_0$ and $C_1$, 
respectively. 
Here, we have $\bm{f}^{(n)} \in \mathscr{X}_T^{m-2}$ and $\dt^{m-1}\bm{f}^{(n)} \in L^1(0,T;L^2)$. 
Moreover, it is straightforward to check that the data $(\bm{y}_0^\mathrm{in},\bm{y}_1^\mathrm{in},\bm{f}^{(n)})$ 
satisfy the compatibility conditions up to order $m-1$. 
Therefore, by Theorem \ref{Th1} there exists a unique solution $\bm{y}\in\mathscr{X}_T^m$ to \eqref{nthLS}. 
We see also that $\bm{y}\in\mathscr{S}_T^m$. 
Now, we define $\bm{y}^{(n+1)}$ as this solution $\bm{y}$. 
In this way, we have constructed a sequence of approximate solutions $\{\bm{y}^{(n)}\}_{n=1}^\infty$.

We proceed to show that $\{\bm{y}^{(n)}\}_{n=1}^\infty$ converges in $\mathscr{X}_T^m$. 
Put $\bm{u}^{(n)}=\bm{y}^{(n+1)}-\bm{y}^{(n)}$ and $\mu^{(n)}=\nu_\mathrm{l}^{(n+1)}-\nu_\mathrm{l}^{(n)}$. 
Then, we see that $\bm{u}^{(n+1)}$ solves 
\[
\begin{cases}
 \ddot{\bm{u}}^{(n+1)}=(A\bm{u}^{(n+1)\prime})'+Q\bm{u}^{(n+1)\prime}(1,t)+(\mu^{(n)}\bm{x}')' &\mbox{in}\quad (0,1)\times(0,T), \\
 \bm{u}^{(n+1)}=\bm{0} &\mbox{on}\quad \{s=1\}\times(0,T), \\
 (\bm{u}^{(n+1)},\dot{\bm{u}}^{(n+1)})|_{t=0}=(\bm{0},\bm{0}) &\mbox{in}\quad (0,1),
\end{cases}
\]
and $\mu^{(n)}$ solves 
\[
\begin{cases}
 -\mu^{(n)\prime\prime}+|\bm{x}''|^2\mu^{(n)}
  = 2\dot{\bm{x}}'\cdot\dot{\bm{u}}^{(n)\prime} - 2(\bm{x}''\cdot\bm{u}^{(n)\prime\prime})\tau &\mbox{in}\quad (0,1)\times(0,T), \\
 \mu^{(n)} = 0 &\mbox{on}\quad \{s=0\}\times(0,T), \\
 \mu^{(n)\prime} = -2\dot{\bm{x}}'\cdot\dot{\bm{u}}^{(n)} + 2(\bm{x}''\cdot\bm{u}^{(n)\prime})\tau &\mbox{on}\quad \{s=1\}\times(0,T).
\end{cases}
\]
We note that $(\dt^j\mu^{(n)})|_{t=0}=0$ for $j=0,1,\ldots,m-2$. 
Therefore, by Propositions \ref{prop:BEE} and \ref{prop:HOEE} we have 
\begin{align*}
&I_{\gamma,T}( \opnorm{\bm{u}^{(n+1)}(\cdot)}_m ) \\
&\leq
\begin{cases}
 C_0S_{\gamma,T}^*( \|\dt(\mu^{(n)}\bm{x}')'(\cdot)\|_{L^2} ) &\mbox{for}\quad m=2, \\
 C_0\{ I_{\gamma,T}( \opnorm{ (\mu^{(n)}\bm{x}')'(\cdot) }_{m-2} ) + S_{\gamma,T}^*( \|\dt^{m-1}(\mu^{(n)}\bm{x}')'(\cdot)\|_{L^2} ) \}
  &\mbox{for}\quad m\geq3.
\end{cases}
\end{align*}
Moreover, by Lemmas \ref{lem:EstDtNu1}, \ref{lem:EstDtNu2}, \ref{lem:EstAu}, and \ref{lem:EstAu2}, we have 
\[
\begin{cases}
 \opnorm{ (\mu^{(n)}\bm{x}')'(\cdot) }_{m-2} \leq C_0\opnorm{\bm{u}^{(n)}(\cdot)}_{m-1} &\mbox{for}\quad m\geq3, \\
 \|\dt^{m-1}(\mu^{(n)}\bm{x}')'(\cdot)\|_{L^2} \leq C_1\opnorm{\bm{u}^{(n)}(\cdot)}_{m} &\mbox{for}\quad m\geq2.
\end{cases}
\]
These estimates together with \eqref{S*} imply 
\begin{align*}
I_{\gamma,T}( \opnorm{\bm{u}^{(n+1)}(\cdot)}_m )
&\leq C_1S_{\gamma,T}^*( \opnorm{\bm{u}^{(n)}(\cdot)}_m ) \\
&\leq C_1\gamma^{-1} I_{\gamma,T}( \opnorm{\bm{u}^{(n)}(\cdot)}_m ).
\end{align*}
Therefore, if we choose $\gamma$ so large that $2C_1\leq\gamma$, then we see that $\{\bm{y}^{(n)}\}_{n=1}^\infty$ converges in $\mathscr{X}_T^m$. 
Let $\bm{y}\in\mathscr{X}_T^m$ be the limit. 
We see also that $\{\nu_\mathrm{l}^{(n)}\}_{n=1}^\infty$ converges to a $\nu_\mathrm{l}$ such that $\nu_\mathrm{l}'\in\mathscr{X}_T^{m-2}$. 
Putting $\nu=\nu_\mathrm{p}+\nu_\mathrm{l}$ with $\nu_\mathrm{p}=-((\bm{g}+2\tau\bm{x}'')\cdot\bm{y}^{\prime})|_{s=1}\phi$, we see that 
$(\bm{y},\nu)$ is the desired solution. 
Moreover, the energy estimate \eqref{EstLP} in Theorem \ref{Th2} can be obtained similarly as above.

It remains to show \eqref{EstNu} so that we assume also Assumption \ref{ass:addxtau2}. 
Similar to the proof of Lemma \ref{lem:EstDtNu2}, by Lemmas \ref{lem:CalIneq1}--\ref{lem:CalIneqY2} we see that 
\begin{align*}
\opnorm{\nu_\mathrm{l}'}_{m-1,*}
&\leq \opnorm{\nu_\mathrm{l}'}_{m-2} + \opnorm{\nu_\mathrm{l}''}_{m-2}^\dag \\
&\lesssim \opnorm{\nu_\mathrm{l}'}_{m-2} + \opnorm{\nu_\mathrm{l}\bm{x}''\cdot\bm{x}''}_{m-2}^\dag + \opnorm{\tau\bm{x}''\cdot\bm{y}''}_{m-2}^\dag 
 + \opnorm{\dot{\bm{x}}'\cdot\dot{\bm{y}}'}_{m-2}^\dag + \opnorm{h}_{m-2}^\dag \\
&\lesssim \opnorm{\nu_\mathrm{l}'}_{m-2} + \opnorm{\bm{y}}_m + \opnorm{h}_{m-2}^\dag.
\end{align*}
Moreover, by \eqref{pnu} and Lemma \ref{lem:EstPhi} we see also that 
\begin{align*}
\opnorm{\nu_\mathrm{p}'}_{m-1,*}
&\lesssim (1+\|(\tau',\ldots,\dt^{m-2}\tau')\|_{L^2} \|(\bm{x},\ldots,\dt^{m-2}\bm{x})\|_{X^3}) \opnorm{\bm{y}}_m \opnorm{\phi'}_{m-1,*} \\
&\lesssim \opnorm{\bm{y}}_m.
\end{align*}
These estimates imply \eqref{EstNu}. 
The continuity in $t$, that is, $\nu'\in\mathscr{X}_T^{m-1,*}$ can be proved by evaluating $\opnorm{\nu'(t_1)-\nu'(t_2)}_{m-1,*}$ in the same way as above. 
The proof of Theorem \ref{Th2} is complete. 
\hfill$\Box$


\bigskip
Tatsuo Iguchi \par
{\sc Department of Mathematics} \par
{\sc Faculty of Science and Technology, Keio University} \par
{\sc 3-14-1 Hiyoshi, Kohoku-ku, Yokohama, 223-8522, Japan} \par
E-mail: \texttt{iguchi@math.keio.ac.jp}

\bigskip
Masahiro Takayama \par
{\sc Department of Mathematics} \par
{\sc Faculty of Science and Technology, Keio University} \par
{\sc 3-14-1 Hiyoshi, Kohoku-ku, Yokohama, 223-8522, Japan} \par
E-mail: \texttt{masahiro@math.keio.ac.jp}

\end{document}